\documentclass[ejs]{imsart}

\RequirePackage[OT1]{fontenc}
\RequirePackage{amsthm,amsmath}
\RequirePackage[numbers]{natbib}
\RequirePackage[colorlinks,citecolor=blue,urlcolor=blue]{hyperref}
\usepackage{graphicx}

\pubyear{2020}
\volume{0}
\issue{0}
\firstpage{1}
\lastpage{8}

\startlocaldefs
\numberwithin{equation}{section}
\theoremstyle{plain}

\endlocaldefs

\usepackage{amsfonts}
\usepackage{mathrsfs}
\usepackage{float}
\usepackage{amsmath,amssymb}
\usepackage{latexsym}
\usepackage{epstopdf}
\newtheorem {theorem}{Theorem}[section]
\newtheorem {corollary}{Corollary}[section]
\newtheorem{lemma}{Lemma}[section]
\newtheorem{remark}{Remark}[section]
\newcommand{\E}{\mathbb{E}}
\newcommand{\bi}[1]{\mbox{\boldmath{$ #1 $}}}

\begin{document}

\begin{frontmatter}
\title{\Large Asymptotic Normality of Gini Correlation in High Dimension with Applications to the K-sample Problem}
\runtitle{Categorical Gini Correlation in High Dimension}

\begin{aug}
\author{\fnms{Yongli} \snm{Sang}\ead[label=e1]{yongli.sang@louisiana.edu}}
\and
\author{\fnms{Xin} \snm{Dang}\thanksref{t}\ead[label=e2]{xdang@olemiss.edu}}

\address{Department of Mathematics, University of Louisiana at Lafayette, Lafayette, LA 70504, USA\\
Department of Mathematics, University of Mississippi, University, MS 38677, USA\\
\printead{e1,e2}}

\thankstext{t}{Corresponding author}
\runauthor{Y. Sang and X. Dang}
\end{aug}

\begin{abstract}
The categorical Gini correlation proposed by Dang $et$ $al.$ \cite{Dang2021} is a dependence measure to characterize independence between categorical and numerical variables.  The asymptotic distributions of the sample correlation under dependence and independence have been established when the dimension of the numerical variable is fixed. However, its asymptotic behavior for high dimensional data has not been explored.  
In this paper, we develop the central limit theorem for the Gini correlation in the more realistic setting where the dimensionality of the numerical variable is diverging. 
 We then construct  a powerful and consistent test for the $K$-sample problem based on the asymptotic normality. The proposed test not only avoids computation burden but also gains power over the permutation procedure.   Simulation studies and real data illustrations show that the proposed test is more competitive to existing methods  across a broad range of realistic situations, especially in unbalanced cases. 
 \end{abstract}


\tableofcontents
\end{frontmatter}

\section{Introduction}
Recently, Dang $et$ $al.$ \cite{Dang2021} proposed the categorical Gini covariance and correlation, $\mbox{gCov}(\bi X, Y)$ and $\mbox{gCor}(\bi X, Y)$, to measure  dependence of a $p$-variate numerical variable $\bi X$  and a categorical variable $Y$.   
Suppose that the categorical variable $Y$ takes values $L_1, ...,L_K$ and its distribution $P_Y$ is $P(Y = L_k) = p_k>0$ for $k=1,2,...,K$. $\bi X$ is from $F$ and $\psi$ denotes its characteristic function. 
Assume that the conditional distribution of $\bi{X}$ given $Y=L_k$  is $F_k$ with the corresponding characteristic function $\psi_k$.  The Gini covariance is defined as 
\begin{equation} \label{eqn:gcov}
\mbox{gCov}(\bi X, Y) = c(p) \sum_{k=1}^K p_k \int_{\mathbb{R}^p} \frac{|\psi_k(\bi t) -\psi(\bi t)|^2}{\|\bi t\|^{p+1}} d\bi t,
\end{equation}
where $c(p)$ is a known constant.  The Gini covariance measures dependence of $\bi X$ and $Y$ by quantifying the difference between the conditional and the unconditional characteristic functions. The corresponding Gini correlation standardizes the Gini covariance to have a range in [0,1].   Zero Gini covariance or correlation mutually implies independence. 

Another dependence measure that could characterize  independence between two random variables is the popular distance correlation proposed by Szekely, Rizzo and Bakirov~\cite{Szekely2007}. It is flexible for $\bi X$ and $\bi Y$ in arbitrary dimensions and any types (numerical or categorical). It has attracted much attention since then,  see e.g. \cite{Gao2021, Li2018, Lyons13, Szekely2009, Szekely2013a, Szekely2013b,  Szekely17, Zhu2020, Zhu2021} and references therein. In the case of $p$-variate $\bi X$ and categorical $Y$, the distance covariance becomes
\begin{equation} \label{eqn:dcov}
\mbox{dCov}(\bi X, Y) = c(p) \sum_{k=1}^K p_k^2 \int_{\mathbb{R}^p} \frac{|\psi_k(\bi t) -\psi(\bi t)|^2}{\|\bi t\|^{p+1}} d\bi t. 
\end{equation}
Comparing (\ref{eqn:gcov}) and (\ref{eqn:dcov}), we see that the two covariances are closely related. When the categorical variable $Y$ takes two values ($K=2$) or $P_Y$ is uniform, they are only different with a scaling factor~\cite{Dang2021}. While for the general $K\geq 3$ and unbalanced $P_Y$, the Gini covariance is a better dependence measure than the distance covariance because the weight $p_k$ in (\ref{eqn:gcov}) takes the nature of the categorical variable, while $\mbox{dCov}(\bi X, Y)$, due to its squared weights, is dominated by the classes with large probabilities and the contribution from smaller classes is substantially reduced.  

A fruitful research has been developed to study the asymptotic distributions of the sample distance statistics in different scenarios. Under independence of $\bi X \in \mathbb{R}^p$ and $\bi Y \in \mathbb{R}^q$, the standard sample distance covariance or correlation converges in distribution to a mixture of chi-squared distribution in the classical setting where the sample size $n \to \infty$ and $p, q$ are fixed~\cite{Szekely2007, Huo2016}; while in the high-dimension-low-sample size (HDLSS) setting when $p, q \to \infty$ and $n$ is fixed,  Sz{\'e}kely and Rizzo \cite{Szekely2013b} derived a $t$-distribution limit of the unbiased sample covariance by assuming that the components of the high-dimensional vectors in $\bi X$ and in $\bi Y$ are exchangeable; Zhu $et$ $al.$ \cite{Zhu2020} relaxed that assumption and considered the high-dimensional medium-sample-size setting (HDMSS) where $n, p, q \to \infty$ but $p, q$ growing more rapidly than $n$;  Gao $et$ $al.$ \cite{Gao2021} have developed central limit theorems in a more realistic high dimensional high-sample-size setting (HDHSS) where $n$ and $p+q$ diverge in an arbitrary fashion. This result applies for the sample distance covariance of (\ref{eqn:dcov}) in which $q =1$ and $n,p \to \infty$. However, there is no literature to study limiting distributions of sample Gini covariance and correlation in high dimension.   

In the classical setting, Dang $et$ $al.$ \cite{Dang2021} studied the asymptotic distributions of the $V$-statistic sample Gini covariance and correlation with the dimension of $\bi X$ fixed.  They admit normal limits when $\bi X$ and $Y$ are dependent and converge in distribution to a quadratic form of centered Gaussian random variables when $\bi X$ and $Y$ are independent. 
In this paper, we will work with unbiased $U$-statistic covariance estimator and associated sample Gini correlation in high dimension. The first objective of this paper is to establish their asymptotic distributions for sample Gini covariance and correlation under independence of $\bi X$ and $Y$  in the HDHSS setting. 

The derived asymptotic distributions can be used for independence test. In other words,  it is to test the equality of $K$ conditional distributions, which is the classical $K$-sample problem encountered in almost every research field.  Due to its fundamental importance and wide applications, research for this $K$-sample problem has been kept active since 1940's.   For example, the widely used and well-studied tests such as Cram\'{e}r-von Mises test \cite{Kiefer1959, Curry2019}, Anderson-Darling test \cite{Darling57, Scholz1987} and their variations utilize different norms on the difference of empirical distribution functions, while some \cite{Anderson94, Martinez09} are based on the comparison of density estimators if the underlying distributions are continuous.  Other tests \cite{Szekely04, Fernandez08} are based on the difference between characteristic functions. Indeed, the test in \cite{Szekely04} is equivalent to ours, but their test only considers the case of $K=2$. Another equivalent test is the DISCO~\cite{Rizzo2010} whose test statistic is the ratio of the between Gini variation and the within Gini variation, while our Gini correlation is the ratio of the between Gini variation and the total Gini variation. Heller, Heller and Gorfine \cite{Heller2013} and Heller $et$ $al.$ \cite{Heller2016} proposed a dependence test based on rank distances. 
All those distance-based tests require a permutation procedure to determine the critical values. Sang, Dang and Zhao \cite{Sang2020} developed a nonparametric test which applies the jackknife empirical likelihood and has a standard limiting chi-squared distribution. Other tests viewing the $K$-sample test as an independent test between a numerical and categorical variable can be found in \cite{Cai2017, Jiang2015, Wang2017}. 
However, most the afore-mentioned work focuses on the fixed dimensional case and perform poorly or may even fail in high dimension. 

Recently, several distance-based tests for two-sample problem have been proposed in high dimension, see \cite{Biswas2014, Chen2017, Li2018, Zhu2021}. Li \cite{Li2018} constructed a test based on interpoint distances under HDLSS. 
Zhu and Shao \cite{Zhu2021} studied the two-sample problem using energy distance (ED) and maximum mean discrepancy with Gaussian and Laplacian kernels under HDLSS and HDMSS, in which they have shown that all these tests are inconsistent under some scenarios. The general $K$-sample testing in high dimension is more challenging and results in literature are very scarce. Mukhopadhyay and Wang \cite{Mukhopadhyay2020} constructed a graph-based nonparametric approach under HDLSS. However, the power for the test is extremely low under some settings. Gao $et$ $al.$ \cite{Gao2021} tested the $K$-sample problem in high dimension based the distance correlation. 

Our second objective of this paper is to use the asymptotic result of the Gini covariance or correlation to construct a new consistent $K$-sample test in high dimension.  The advantages of the new test include
\begin{itemize}
\item computational efficiency. It avoids a permutation procedure which is computationally expansive. 
\item statistical efficiency. It gains power over the nonparametric tests.
\item robustness for class imbalance. It is more appropriate than the distance based tests in dealing with unbalanced data.  
\end{itemize}

Throughout this paper, if not mentioned otherwise, the letter $C$, with or without a subscript, denotes a generic positive finite constant whose exact value is independent of sample sizes and may change from line to line. $\|\cdot\|$ represents the Euclidean norm, that is, $\| \bi a\|=\sqrt{a^2_1+a^2_2+\cdots+a^2_p}$ for a $p$-vector $\bi a=(a_1, a_2, \cdots, a_p)^T \in \mathbb{R}^p$. For two  sequences, $a_n, b_n$, of real numbers, $a_n=o(b_n)$ means $\lim_{n \to \infty}{a_n}/{b_n}=0$, and $a_n=O(b_n)$ means $L \leq {a_n}/{b_n} \leq U$ for some finite constants $L$ and $U$. For random variable sequences, similar notations $o_p(n)$ and $O_p(n)$ are used to stand for the relationships holding in probability. 

The remainder of the paper is organized as follows. In Section \ref{sec:cgc}, we first briefly review the other representation of Gini correlation and review the existing statistical inference, then we present a $U$-estimator for the Gini correlation and the central limit theorem for the $U$-estimator when both the sample sizes and dimensionality are diverging.  The $K$-sample test is proposed and its consistency is established. In Section \ref{sec:simulationstudy}, we conduct simulation studies to evaluate the performance of the proposed test. A real data analysis is illustrated in Section \ref{sec:realdata} to compare the proposed test with current existing approaches. 
We conclude and discuss future works in Section \ref{sec:conclusion}. All technical proofs are provided in Appendix.

\section{Inference of Gini covariance and correlation in high dimension}\label{sec:cgc}
\subsection{Categorical Gini correlation}
The Gini covariation between $\bi X$ and $Y$ defined in (\ref{eqn:gcov}) can be represented in the multivariate Gini mean differences (GMD). Let $(\bi X_1,\bi X_2)$ and $(\bi X_1^{(k)}, \bi X_2^{(k)})$ be independent pair variables independently from $F$ and $F_k$ respectively. Define $\Delta =\E\|\bi X_1-\bi X_2\|$ as the GMD of $F$ and $\Delta_k=\E \|\bi X_1^{(k)}-\bi X_2^{(k)}\|$ as the GMD of $F_k$. From \cite{Dang2021}, we have
\begin{equation}\label{eqn:gcov2}
\mbox{gCov}(\bi X,Y) = \Delta-\sum_{k=1}^Kp_k\Delta_k.
\end{equation}
Then the Gini correlation is 
\begin{equation} \label{eqn:gcor}
\mbox{gCor}(\bi X, Y) = \frac{ \Delta-\sum_{k=1}^Kp_k\Delta_k}{\Delta}. 
\end{equation}
This representation allows a nice interpretation. The Gini covariance is the between Gini variation and the Gini correlation is the ratio of the between and the total variation. Also from this representation, it is straightforward to obtain sample estimators.

Dang {\em et al.} \cite{Dang2021} used V-statistic estimators and derived limiting distributions of the estimators under the classical setting when the dimension of $\bi X$ is fixed. More specifically, 
suppose a sample  ${\cal D} =\{(\bi X_1, Y_1), (\bi X_2, Y_2), ...., (\bi X_n, Y_n)\}$ is drawn from the joint distribution of $\bi X$ and $Y$. Write ${\cal D} ={\cal D}_1\cup {\cal D}_2...\cup {\cal D}_K$,  where ${\cal D}_k=\left \{\bi X^{(k)}_{1}, \bi X^{(k)}_{2}, ...,\bi X^{(k)}_{n_k}\right \}$ is the sample with $Y_i=L_k$ and $n_k$ is the number of sample observations in the $k^{th}$ class.  Plugging in $\hat{p}_k= n_k/n$ and V-statistics 
 \begin{align*} 
&  \;\; \tilde{\Delta}_k = n_k^{-2}\sum_{1 \leq i,j \leq n_k} \|\bi X_i^{(k)} -\bi X_j^{(k)}\|,  \;\;\; \tilde{\Delta} = n^{-2} \sum_{1\leq i,j \leq n} \|\bi X_i -\bi X_j\|
\end{align*}
to (\ref{eqn:gcor}), the estimator $\hat{\rho}_g(\bi X,Y)$ is obtained.  Under the assumption of $\E \|\bi X\|^2 < \infty$ with $p$ fixed and $n \to \infty$, $\hat{\rho}_g(\bi X,Y)$ has the limiting distributions as below. 

\begin{enumerate}
\item If $\mbox{gCor}(\bi X,Y)\neq 0$, then  $$ \sqrt{n} (\hat{\rho}_g(\bi X,Y) -\mbox{gCor}(\bi X,Y)) \stackrel{d}{\longrightarrow} {\cal N}(0, \sigma_g^2),$$
where $\sigma_g^2$ is the asymptotic variance. 
\item If $\mbox{gCor}(\bi X,Y)=0$, then
\begin{equation} \label{eqn:qchisq}
n \hat{\rho}_g(\bi X,Y)   \stackrel{d}{\longrightarrow} \frac{4}{\Delta}\left [ \sum_{s=1}^\infty \sum_{k=1}^K (1-p_k)  \lambda_s Z_{s,k}^2+ \sum_{s=1}^\infty \sum_{1\leq k <l\leq K} \sqrt{p_kp_l} \lambda_s Z_{s,k}Z_{s,l}\right ],
\end{equation}
where $Z_{s,k} (k=1,...,K, s=1,2,...)$ are independent standard normal variates and $\lambda_s$ are nonnegative coefficients.  
\end{enumerate}

Under  independence of $\bi X$ and $Y$,  $ \hat{\rho}_g(\bi X,Y)$ converges to a quadratic form of normal random variables. This result is difficult to be applied for the independence test, and hence one has to rely on a permutation procedure to determine a critical value of the test, which is computationally expansive. 

This result is obtained under the classical setting.  The inference for the Gini correlation in high dimension has not been explored and we will fill this gap by developing the asymptotic distributions when both the sample sizes and the dimensionality diverge to infinity. 

\subsection{U-estimators and projection representation}\label{sec:hdcgc}
When the dimension $p$ is large, the V-statistic Gini covariance and correlation estimators may have  issues about bias. Therefore,  we will estimate the GMDs  by unbiased $U$-statistics. That is, 
\begin{align} 
& \hat{\Delta}= {n \choose 2}^{-1}\sum_{1 \leq i <j \leq n}\|\bi X_i-\bi X_j\|:=U_n; \nonumber\\
 & \hat{\Delta}_k = {n_k \choose 2}^{-1}\sum_{1 \leq i<j \leq n_k}\|\bi X^{(k)}_i-\bi X^{(k)}_j\|:=U_{n_k}.\label{Ustat}
\end{align}
Thus Gini covariance and correlation can be estimated by
\begin{align}\label{gcovn}
\mbox{gCov}_n(\bi X, Y)&= \hat{\Delta} - \sum_{k=1}^K \hat{p}_k \hat{\Delta}_k = U_n -  \sum_{k=1}^K \hat{p}_k U_{n_k}
\end{align}
and 
\begin{align}\label{gcor}
\mbox{gCor}_n(\bi X, Y)&=\dfrac{\mbox{gCov}_n(\bi X, Y)}{\hat{\Delta}} = \dfrac{U_n -  \sum_{k=1}^K \hat{p}_k U_{n_k}}{U_n},
\end{align}
respectively. Both of them are functions of $U$-statistics $U_n$ and  $U_{n_k}$'s. We shall focus on the asymptotic distribution of $\mbox{gCov}_n(\bi X, Y)$. The application of Slutsky's theorem allows us to obtain the result on $\mbox{gCor}_n(\bi X, Y)$ immediately. 

Under independence of $\bi X$ and $Y$, the sample Gini covariance $\mbox{gCov}_n$ in (\ref{gcovn}) is a linear combination of $U$-statistics with first-order degeneracy.  By classical theory about $U$ statistics in the fixed dimensional asymptotic (fixed dimension with sample sizes diverge to infinity), a non-normal limiting distribution holds, a similar result as (\ref{eqn:qchisq}).  However, as both the  the dimension and the sample size go large, the degenerate $U$-statistic will admit a normal limit. To establish this result,  we first take decompositions of U-statistics in (\ref{Ustat}) and rewrite (\ref{gcovn}).   

By the Hoeffding decomposition, we have
\begin{align*}
U_n = \Delta + \dfrac{2}{n}\sum_{i=1}^n \left \{\E\big(\|\bi X-\bi X_i\||\bi X_i\big)-\Delta\right\}+ {n \choose 2}^{-1}\sum_{1 \leq i<j \leq n}d(\bi X_i, \bi X_j),
\end{align*}
where
\begin{eqnarray}
d(\bi X_1, \bi X_2)&=&\|\bi{X}_1-\bi{X}_2\|-\E\big(\|\bi X_1-\bi X_2\| \bigm |\bi X_1\big)-\E\big(\|\bi X_1-X_2\| \bigm |\bi X_2 \big)\nonumber \\
&&+\E\|\bi X_1-\bi X_2\|\label{d} 
\end{eqnarray}
is called the double centered distance and it is actually the second order centered projection of the kernel function of $U_n$.  Analogously,
\begin{align*}
U_{n_k} = & \Delta_k + \dfrac{2}{n_k}\sum_{i=1}^{n_k} \left \{\E\big(\|\bi X^{(k)}-\bi X_i^{(k)}\||\bi X_i^{(k)}\big)-\Delta_k\right\}\\
&+ {n_k \choose 2}^{-1}\sum_{1 \leq i<j \leq n_k}d(\bi X^{(k)}_i, \bi X^{(k)}_j).
\end{align*}

Under  independence of $\bi X$ and $Y$, we have $F_1 = F_2 =...=F_K =F$.  Hence $ \Delta=\Delta_k, \ k=1,2,...,K$ and 
\begin{align*}
\sum_{k=1}^K \hat{p}_k \dfrac{2}{n_k}\sum_{i=1}^{n_k} \E\big(\|\bi X^{(k)}-\bi X^{(k)}_i\||\bi X^{(k)}_i\big)&
=\dfrac{2}{n}\sum_{i=1}^n \E\big(\|\bi X-\bi X_i\||\bi X_i\big).
\end{align*}
Then we can represent (\ref{gcovn}) as
\begin{align}
&\mbox{gCov}_n(\bi X, Y) \nonumber \\
&={n \choose 2}^{-1}\sum_{1 \leq i<j \leq n}
d(\bi X_i, \bi X_j)-\sum_{k=1}^K \hat{p}_k{n_k \choose 2}^{-1}\sum_{1 \leq i<j \leq n_k}d(\bi X^{(k)}_i, \bi X^{(k)}_j), \label{dgcovn}
\end{align}
under the null that $\bi X$ and $Y$ are independent. 

The representation of (\ref{dgcovn}) has advantages over (\ref{gcovn}) due to appealing orthogonal properties of $d(\bi X_1, \bi X_2)$ as stated in Lemmas \ref{prop_d} and \ref{prop_eta} in Appendix. Those properties largely simplify the calculation of specific moments involved. 

\subsection{Asymptotic normality}\label{subsec:clt}
We study the asymptotic distributions of the $U$-estimators in this section. Let $\bi X_1, \bi X_2, \bi X_3$ and $\bi X_4$ be i.i.d copies of $\bi X$.  The following conditions will be needed to facilitate the proofs. 
\begin{description} \label{Condition}
\item \textbf{C}1. $\E \|\bi X\|^{4} < \infty$;
\item  \textbf{C}2. $\dfrac{\E d^4(\bi X_1, \bi X_2)}{n \big(\E d^2(\bi X_1, \bi X_2)\big)^2} \rightarrow 0$;  
  \item \textbf{C}3. $\dfrac{\E d(\bi X_1, \bi X_3)d(\bi X_2, \bi X_3)d(\bi X_1, \bi X_4)d(\bi X_2, \bi X_4)}{\big(\E d^2(\bi X_1, \bi X_2)\big)^2} \to 0$; 
 \item \textbf{C}4.  $\sqrt{n}\mbox{gCov}(\bi X, Y) \to \infty$.
\end{description}

 \begin{remark}
Our conditions \textbf{C}2 and \textbf{C}3 are corresponding to conditions (18) and (19) in \cite{Gao2021} when $\tau=1$. 
 In fact, the condition \textbf{C}2 can be weaken to be $\dfrac{\E \big( |d(\bi X_1,\bi X_2)|^{2+2\alpha}\big)}{n^\alpha \big(\E d^2(\bi X_1, \bi X_2)\big)^2} \rightarrow 0$ for some constant $0<\alpha \leq 1$. However,  it is hard to check the condition when $0<\alpha<1$, so we take the stronger but simple condition. 
 \end{remark}
Applying Martingale central limit theorem, we establish the limiting distribution of the sample Gini covariance in the following theorem.  
\begin{theorem}\label{gcov1}
 Under  independence of $\bi X$ and $Y$, and conditions \textbf{C}1-\textbf{C}3,  as $\min\{n_1, n_2, ..., n_k\} \rightarrow \infty$ and $p \rightarrow \infty$, we have
\begin{align*}
\dfrac{\mbox{gCov}_n(\bi X, Y)}{\sigma_0} \stackrel{{d}}{\longrightarrow} {\cal N}(0,1),\ 
\end{align*}
where $\sigma^2_0 =(\sum_{k=1}^K \hat{p}^2_k {n_k \choose 2}^{-1}-{n\choose 2}^{-1})\E d^2(\bi X_1,\bi X_2)$ is the variance of $\mbox{gCov}_n(\bi X, Y)$.
\end{theorem}
Theorem \ref{gcov1} reveals that a degenerate $U$-statistic admits a normal limit due to the high dimensionality. 
This is surprisingly inspiring to deal with problems which can be estimated by $U$-statistics in high dimension. 
%
%
 
To make inference feasible,   we need to estimate $\sigma_0^2$.  A consistent estimator $\hat{\sigma}^2_0$ is 
\begin{align}\label{eqn:sigma0}
\hat{\sigma}^2_0=\bigg(\sum_{k=1}^K \hat{p}^2_k {n_k \choose 2}^{-1}-{n\choose 2}^{-1}\bigg)V^{2}_n(\bi X),
\end{align}
where  $V^{2}_n(\bi X)$ is the bias-corrected estimator for the squared distance variance in \cite{Szekely2013b}. That is,
\begin{align*}
V^{2}_n(\bi X)=\dfrac{1}{n(n-3)}\sum_{1 \leq k\neq l\leq n}A^{2}_{k,l}
\end{align*}
with $A_{k,l}$ being the centered sample distance, which is 
\begin{align*}
A_{k,l}=&\|\bi X_k-\bi X_l\|-\dfrac{1}{n-2}\sum_{i=1}^n \|\bi X_i-\bi X_l\|-\dfrac{1}{n-2}\sum_{j=1}^n \|\bi X_k-\bi X_j\|\\
&+\dfrac{1}{(n-1)(n-2)}\sum_{1 \leq i, j \leq n}\|\bi X_i-\bi X_j\|.
\end{align*}
\begin{theorem}\label{gcov2}
Under  independence of $\bi X$ and $Y$, and conditions \textbf{C}1-\textbf{C}3,  as $\min\{n_1, n_2, ..., n_k\} \rightarrow \infty$ and $p \rightarrow \infty$, we have
\begin{align*}
\dfrac{\mbox{gCov}_n(\bi X, Y)}{\hat{\sigma}_0} \stackrel{{d}}{\longrightarrow} {\cal N}(0,1).
\end{align*}
\end{theorem}

The estimators in (\ref{Ustat}) are $U$-statistics and hence the ratio is consistent with $\hat{\Delta}/\Delta \to 1$ in probability. By  applying Slutsky's theorem, we have the CLT for the Gini correlation. 
\begin{corollary}\label{gcor_an}
 Under  independence of $\bi X$ and $Y$, and conditions \textbf{C}1-\textbf{C}3,  as $\min\{n_1, n_2, ..., n_k\} \rightarrow \infty$ and $p \rightarrow \infty$, we have
\begin{align*}
\dfrac{\hat{\Delta}}{\hat{\sigma}_0} \mbox{gCor}_n(\bi X, Y) \stackrel{{d}}{\longrightarrow} {\cal N}(0,1).
\end{align*}
\end{corollary}

From the result of (\ref{eqn:qchisq})  and  Corollary \ref{gcor_an}, we see that when $\bi X$ and $Y$ are independent, as the dimensionality of the numerical variable goes large and under some conditions on the fourth moment,  the complicate quadratic form of normal distributions converges to a normal distribution.


\subsection{High-dimensional K-sample test}

These established CLTs  can be applied to test  the independence of $\bi X$ and $Y$.  We will use the CLT for the Gini covariance to do the test. The one based on the Gini correlation is asymptotically equivalent. 

The independence test is stated as
\begin{equation}\label{test}
 {\cal H}_0: \mbox{gCov}(\bi X, Y) = 0,\;\;\;\; \mbox{vs}\;\;\;\; {\cal H}_1: \mbox{gCov}(\bi X, Y) >  0. 
 \end{equation}
Note that the null hypothesis of the test  in (\ref{test})  is equivalent to the null of the $K$-sample test   
 \begin{equation*}\label{Ktest}
 {\cal H}_0^\prime: F_1 = F_2 =...=F_K=F. 
 \end{equation*}
In the $K$ sample test, we can view sample point $(\bi X_i,  Y_i)$ in such way. $Y_i$ is the class label of $\bi X_i$.  $Y_i=L_k$ indicates  that $\bi X_i$ is drawn from $F_k$.  The pooled sample ${\cal D} ={\cal D}_1\cup {\cal D}_2...\cup {\cal D}_K$ has the distribution $F$, which is the average distribution of $F_k$'s.

By Theorem \ref{gcov1}, we can
reject ${\cal H}_0$ or ${\cal H}_0^\prime$  if $\mbox{gCov}_n(\bi X, Y)>Z_{\alpha}\hat{\sigma}_0$ at level $\alpha$, where $Z_{\alpha}$ is the $(1-\alpha)100\% $ percentile of the standard normal distribution. 

For $K=2$, the two sample problem, the proposed test is asymptotically equivalent to the test based on distance covariance because $gCov(\bi X,Y) = dCov(\bi X,Y)/\sqrt{dCov(Y,Y)}$. This is the result of Remark 9 in \cite{Dang2021}.  And hence two test statistics estimate a same population quantity.  They are also asymptotically equivalent to Sz\'{e}kely's energy test \cite{Szekely04, Baringhaus2004} that is based on energy statistic between $F_1$ and $F_2$.

Theorem \ref{gcov1} allows us to avoid computation burden of the permutation tests. As demonstrated in the simulation, the test based on the limiting normality is more powerful than the permutation tests. The power function for the proposed test is 
\begin{align*}
P_n(\alpha)=P(\mbox{gCov}_n(\bi X, Y)>Z_{\alpha}\hat{\sigma}_0\bigm| {\cal H}_1).
\end{align*}
The test consistency is established in the below theorem. 
\begin{theorem}\label{power}
For any alternative ${\cal H}_1$  satisfying conditions \textbf{C}1 and \textbf{C}4,   as  $\min\{n_1, n_2, ..., n_k\} \rightarrow \infty$, $p_k>0$ and $p \rightarrow \infty$, we have
 $$P_n(\alpha)=P(\mbox{gCov}_n(\bi X, Y)>Z_{\alpha}\hat{\sigma}_0\bigm| {\cal H}_1) \to 1.$$
\end{theorem}
Condition  \textbf{C}1 is the usual assumption on the finite fourth moment.  Condition \textbf{C}4, $\sqrt{n}gCov(\bi X,Y) \rightarrow \infty$, requires dependence of $\bi X$ and $Y$ cannot be too weak. We might state a local alternative as 
\begin{align*}
{\cal H}_1^\prime: \mbox{gCov}(\bi X, Y) \geq C n^{-t}, \;\;\;\mbox{ for } t < 1/2. 
\end{align*}
The proposed test is able to detect the dependence under ${\cal H}_1^\prime$ with power going to 1 as sample sizes increase.

\section{Simulation study}\label{sec:simulationstudy}
In this section, we conduct three simulation studies to verify the theoretical properties of the standardized Gini covariance statistic and compare its performance in $K$-sample tests with others. 
\begin{figure*}[thb]
\centering
\begin{tabular}{cc}\vspace{-0.2in}
\includegraphics[width=2.4in,height=2.4in]{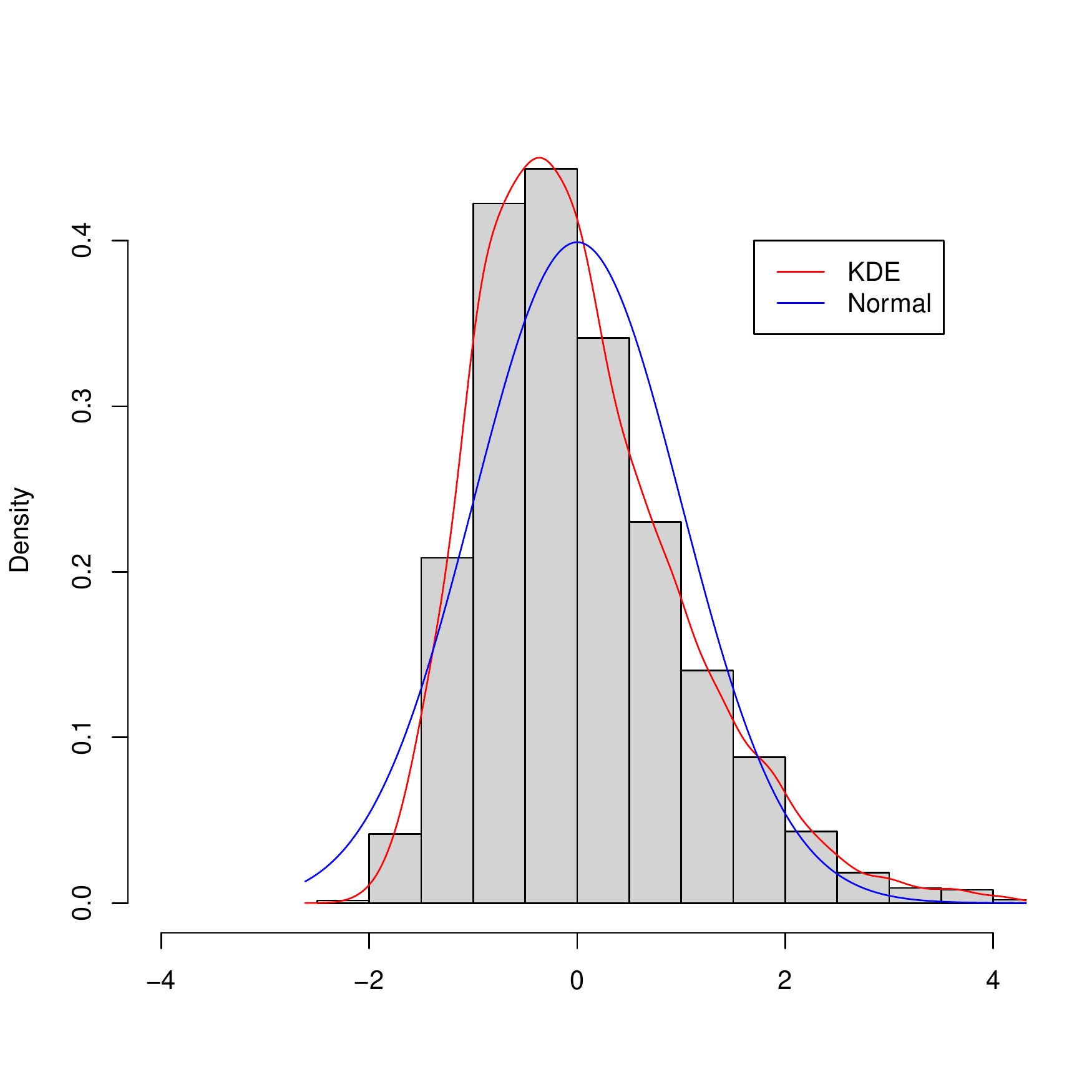} &
\includegraphics[width=2.4in,height=2.4in]{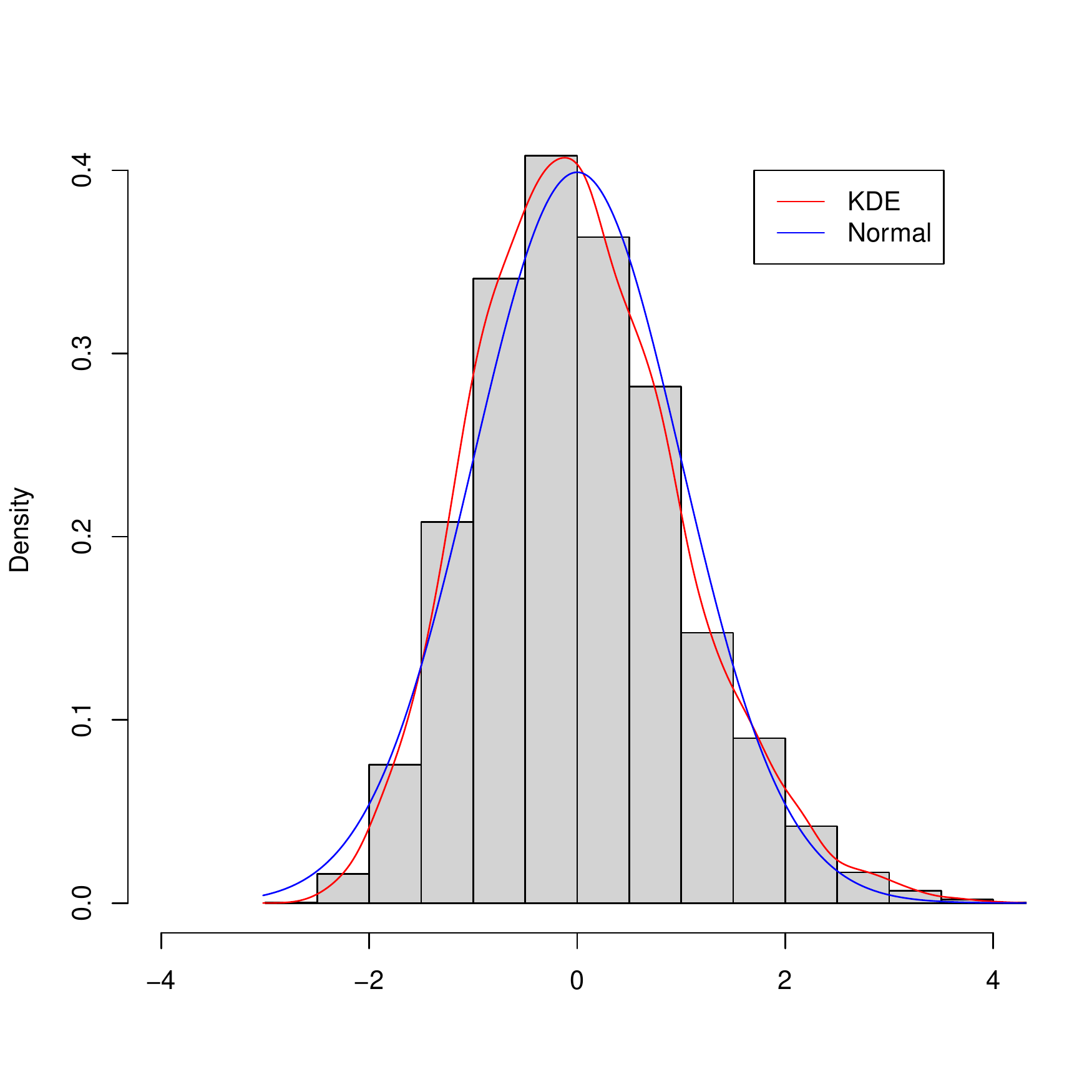}\\ 
(a) $ p=5$ & (b) $p=50$\\ \vspace{-0.2in}
\includegraphics[width=2.4in,height=2.4in]{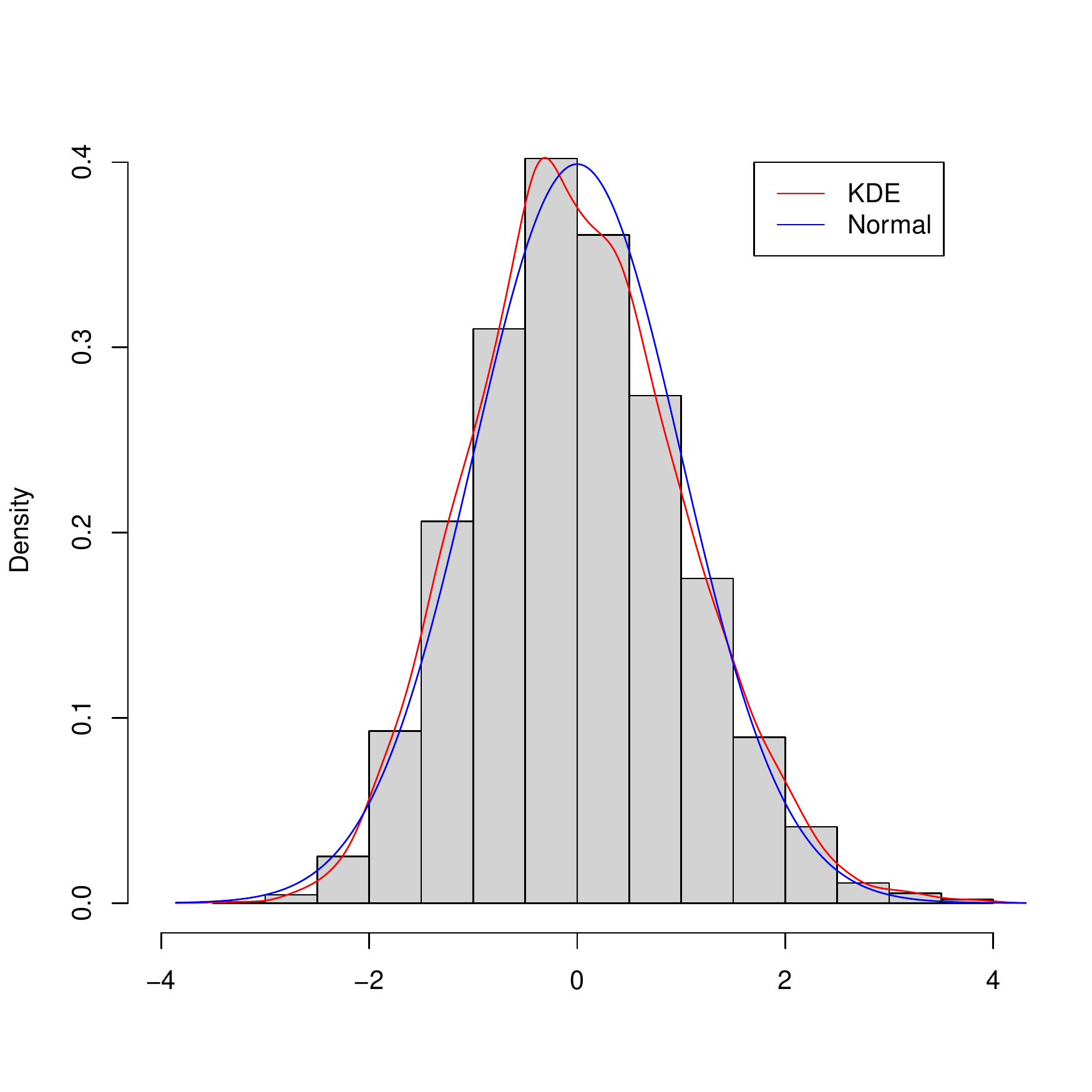} &
\includegraphics[width=2.4in,height=2.4in]{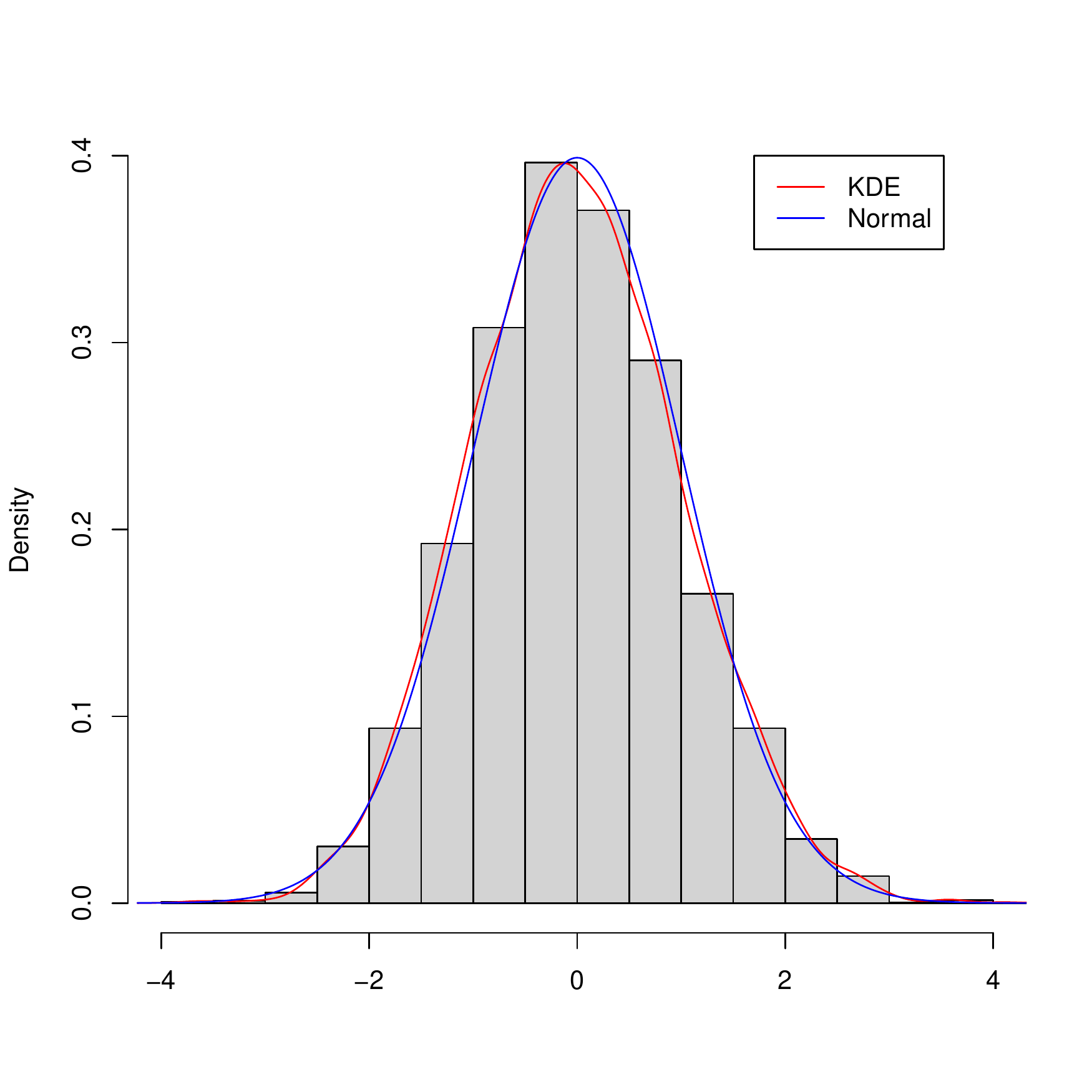}\\ 
(c) $p=200$ & (d) $p=500$
\end{tabular} \vspace{-0.1in}
\caption{Histograms of the standardized Gini covariance statistic in Example 1 with kernel density estimation curves in red and standard normal density curves in blue.  }
\label{fig:hist}
\end{figure*}

\subsection{Limiting normality}
We generate independent $K$ samples from the same multivariate normal distributions and compute the standardized Gini covariance statistic. The procedure is repeated 5000 times. The setup parameters are listed below. 
\begin{description}
\item [Example 1.]  $K=5$ samples of sizes $\bi{n}=(30,40,50,60,70)$ are generated from ${\cal N}_p(\bi 0, \bi\Sigma)$, where  $p=5, 50, 200, 500$ and $\bi \Sigma = (\Sigma_{ij}) \in \mathbb{R}^{p\times p}$ with $\Sigma_{ij} = 0.7^{|i-j|}$.  
\end{description}
For each dimension $p$, the histogram of 5000 standardized Gini covariance statistics is plotted in Figure \ref{fig:hist}.  Also the kernel density estimation (KDE) curve and the standard normal density curve are added to the histogram plot to visualize closeness between empirical density and asymptotical density functions. For $p=5$ in Figure \ref{fig:hist}(a),  the histogram is slightly right-skewed and there is some discrepancy between KDE and the normal curve. But when dimension increases, the discrepancy becomes less and diminishes as shown in Figure \ref{fig:hist}(b)-(d). We also calculate the maximum point distance between KDE and Normal density function as a measure of discrepancy in Table \ref{tab:dist}. It is clear that the difference decreases with dimensionality. Gao $et$ $al.$ \cite{Gao2021} developed the limiting normal distribution for distance correlation, so we also involve the maximum point distance between KDE for the distance measure.
Comparing with the scaled distance covariance statistic, the Gini one has a better normal approximation in each dimension.

\begin{table}[thb]
\centering
\begin{tabular}{c|cccr} \hline\hline
 Distance & $p=5$& $p=50$ &$p=200$ &$p=500$ \\ \hline
Dist(KDE$_g$, Normal) &0.1176 &0.0478& 0.0294& 0.0177\\
Dist(KDE$_d$, Normal) &0.1290 &0.0493& 0.0338& 0.0207\\ \hline\hline
\end{tabular}
\caption{The maximum point distances between the kernel density estimation function and standard normal density function. KDE$_g$ is for rescaled $gCov_n$ and KDE$_d$ for $dCov_n$. } \label{tab:dist}
\end{table}

 \subsection{Size and power in K-sample tests} \label{subsec: test}
 In this simulation, we compare five methods for $K$ sample problem. Two of them are permutation tests. The one based on distance covariance in high dimension has been studied in \cite{Zhu2021} for $K=2$. Here we examine both permutation tests for $K$ sample problem in high dimension. Five methods are
 
 \begin{description}
 \item [gCov:]  our proposed method using rescaled Gini covariance statistic and the normal percentile as the critical value. 
 \item [gCov-perm:] permutation test using Gini covariance statistic. This test is  asymptotically equivalent to the one-way DISCO method \cite{Rizzo2010}. 
 \item [dCov:] the method using rescaled distance covariance statistic using the percentile of the standard normal as the critical value \cite{Gao2021}.
 \item [dCov-perm:]  permutation test using distance covariance statistic.  
 \item [GLP:] graphic LP polynomial basis function method proposed in \cite{Mukhopadhyay2020}.  
 \end{description}

 We consider $K=3$ case in dimensions $p=200, 500$ with the equal size $\bi n = (40,40,40)$, slightly unbalanced size $\bi n = (50, 40, 30)$ and heavily unbalanced size $\bi n = (72, 36, 12)$. Let 
 \begin{align*}
 & \bi \mu_1 =  \bi 0_p; \; {\bi \Sigma}_1={\bi \Sigma} = (\Sigma_{ij}) \in \mathbb{R}^{p \times p},  \mbox{ where } \Sigma_{ij} = 0.7^{|i-j|};\\ 
 &\bi \mu_2 = (0.1\times\bi 1^T_{\beta p}, \bi 0^T_{(1-\beta)p})^T; \;\bi \Sigma_2={\bi D}_1\bi \Sigma \bi D_1\mbox{ with } \bi D_1 = \mbox{diag}(1.1\times \bi 1^T_{\beta p}, \bi 1^T_{(1-\beta)p});  \\
&\bi \mu_3 = (0.2\times \bi1^T_{\beta p}, \bi 0^T_{(1-\beta)p})^T;\; \bi \Sigma_2={\bi D}_2\bi \Sigma \bi D_2\mbox{ with } \bi D_2 = \mbox{diag}(1.2\times \bi 1^T_{\beta p}, \bi 1^T_{(1-\beta)p}). 
 \end{align*}
Here $\beta \in [0,1]$ is the proportion of the $p$ components for which 3 samples differ in mean and in variance. 
 \begin{table}[thb] 
\centering
\begin{tabular}{l  c c|c c c c c c} \hline\hline
 $p$ & $\bi n$ & method & $\beta=0$ &$\beta=.2$ &$\beta=.4$ &$\beta=.6$ &$\beta=.8$ &$\beta=1$ \\ \hline
 200 & (40,40,40)& gCov &.052 &.171&.421&.692&.864&.966\\
&&gCov-perm& .050&.123&.327&.609&.815&.942\\
&&dCov& .053 & .172&.423&.691&.867&.966\\
&&dCov-perm& .043&.159&.416&.665&.852&.954\\
&&GLP & .060 &.098&.254&.466&.720&.875\vspace{0.2cm}\\

& (50,40,30) & gCov &.065 &.183&.484&.718&.882&.949\\
&&gCov-perm& .061&.133&.402&.621&.823&.914\\
&&dCov& .068 & .170&.454&.699&.873&.948\\
&&dCov-perm& .062&.160&.417&.664&.858&.948\\
&&GLP & .069 &.096&.241&.455&.687&.845\vspace{0.2cm}\\

 &(72,36,12) &  gCov &.058 &.155&.282&.476&.632&.814\\
&&gCov-perm& .049&.110&.212&.391&.555&.749\\
&&dCov& .063 & .112&.233&.444&.606&.802\\
&&dCov-perm& .060&.104&.215&.419&.571&.780\\
&&GLP & .066 &.090&.178&.264&.403&.577\\  \hline

500 & (40,40,40)& gCov &.061 &.268&.665&.942&.997&1.00\\
&&gCov-perm& .063&.207&.587&.904&.993&1.00\\
&&dCov& .063 & .274&.667&.943&.997&1.00\\
&&dCov-perm& .060&.269&.654&.934&.998&1.00\\
&&GLP & .049 &.143&.455&.812&.971&.999 \vspace{0.2cm}\\

& (50,40,30) & gCov &.052 &.340&.801&.972&.997&.999\\
&&gCov-perm& .055&.280&.727&.950&.990&.999\\
&&dCov& .058 & .313&.776&.961&.995&.999\\
&&dCov-perm& .051&.308&.762&.956&.994&.998\\
&&GLP & .059 &.156&.428&.800&.956&.993\vspace{0.2cm}\\

 &(72,36,12) &  gCov &.051 &.231&.493&.769&.923&.979\\
&&gCov-perm& .055&.154&.399&.671&.901&.968\\
&&dCov& .054 & .175&.426&.721&.916&.978\\
&&dCov-perm& .052&.172&.420&.711&.909&.976\\
&&GLP & .047 &.109&.240&.450&.688&.853\\ \hline \hline 
\end{tabular}
\caption{Size and Power of Tests for $K=3$ samples in Example 2. }
\label{tab:k3norm}
\end{table}

 \begin{description}
\item [Example 2.]  Generate samples of $\bi X^{(1)} \sim {\cal N}_p(\bi \mu_1, \bi \Sigma_1)$, $\bi X^{(2)} \sim {\cal N}_p(\bi \mu_2, \bi \Sigma_2)$ and $\bi X^{(3)} \sim {\cal N}_p(\bi \mu_3, \bi \Sigma_3)$.  
 \end{description}
 We conduct 1000 simulations. The size and power of each test are computed and reported in Table \ref{tab:k3norm}. The column $\beta=0.0$ corresponds to the size of tests. Several observations can be drawn. All tests maintain the nominal level 5\% quite well. Permutation tests are slightly less powerful than their corresponding counterparts. GLP test is inferior to others in all cases. In the equal size case, Gini method $gCov$ produces almost the same size and power as $dCov$, which is an expected result since the Gini covariance and distance covariance are asymptotically equivalent. While in the unbalanced cases, our Gini method gains 1\% - 6\% power advantage over the distance one.  An intuitive interpretation of the advantage is that $gCov$ is a better measure than $dCov$ in unbalanced distributions as stated in the Introduction section.  
\begin{figure*}[thb]
\centering
\begin{tabular}{cc}\vspace{-0.1in}
\includegraphics[width=2.35in,height=2.35in]{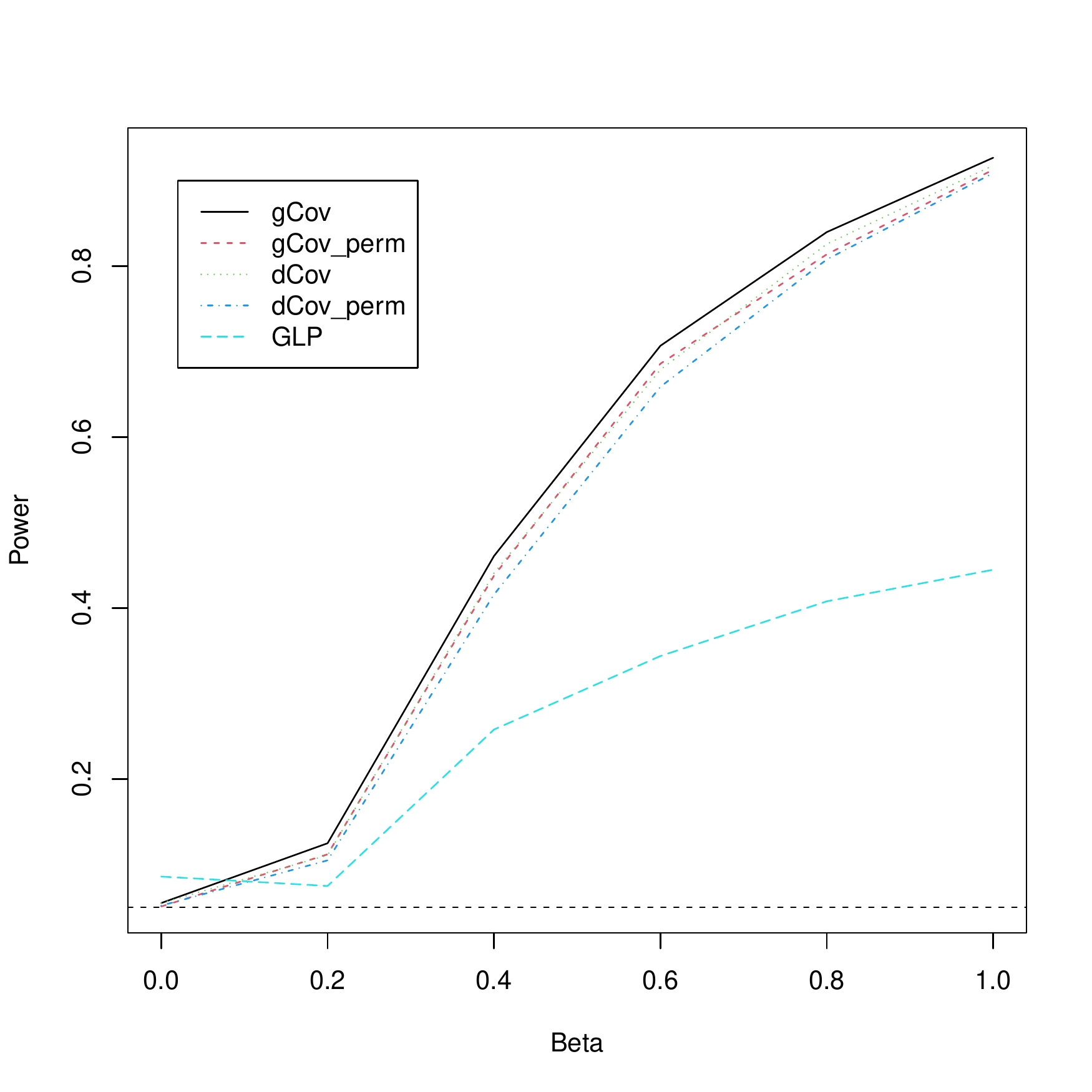} &\hspace{-0.1in}
\includegraphics[width=2.35in,height=2.35in]{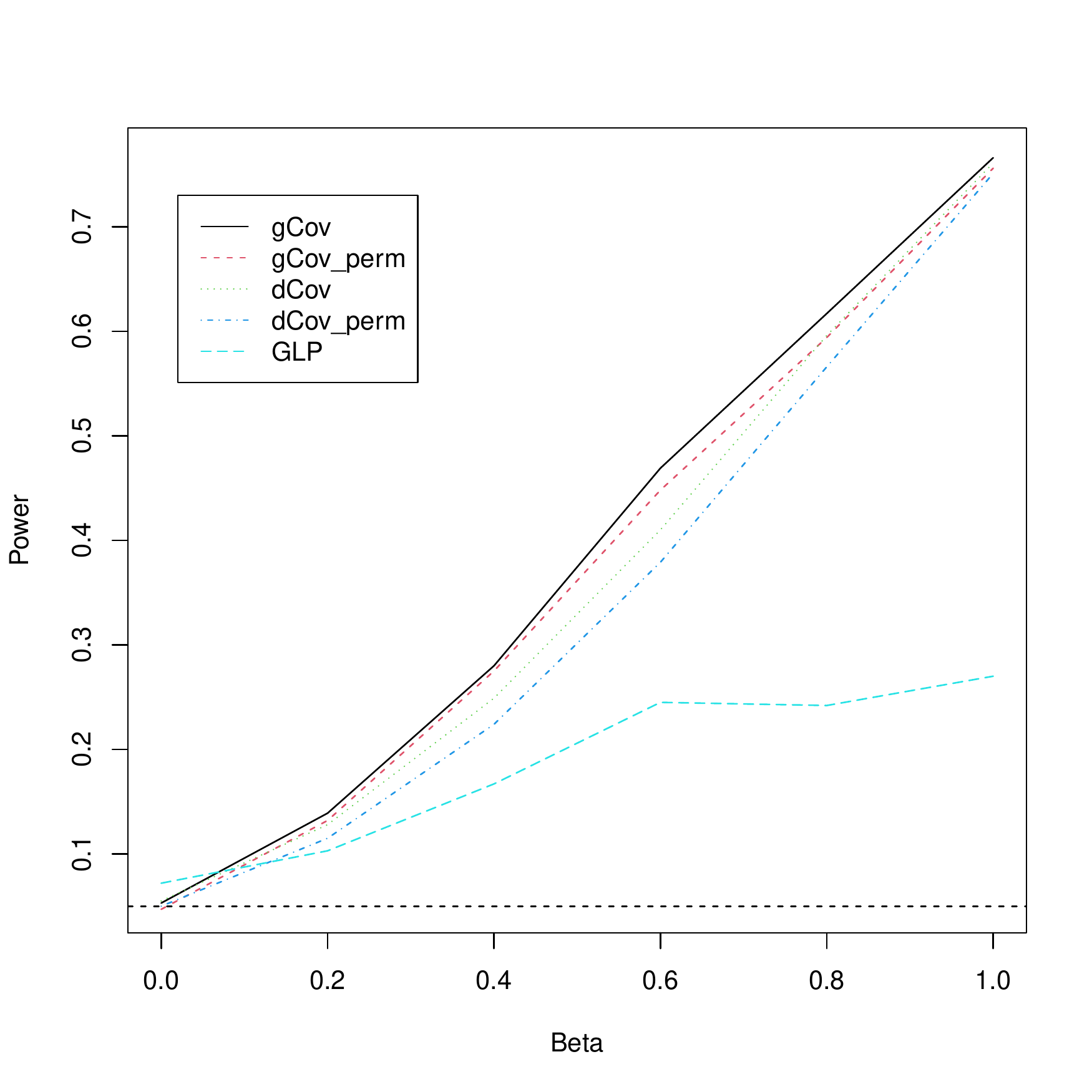}\\ 
(a) $p=200, \bi n=(50,40,30)$ & \hspace{-0.1in}(b) $p=200, \bi n=(72,36,12)$\\ \vspace{-0.1in}
\includegraphics[width=2.35in,height=2.35in]{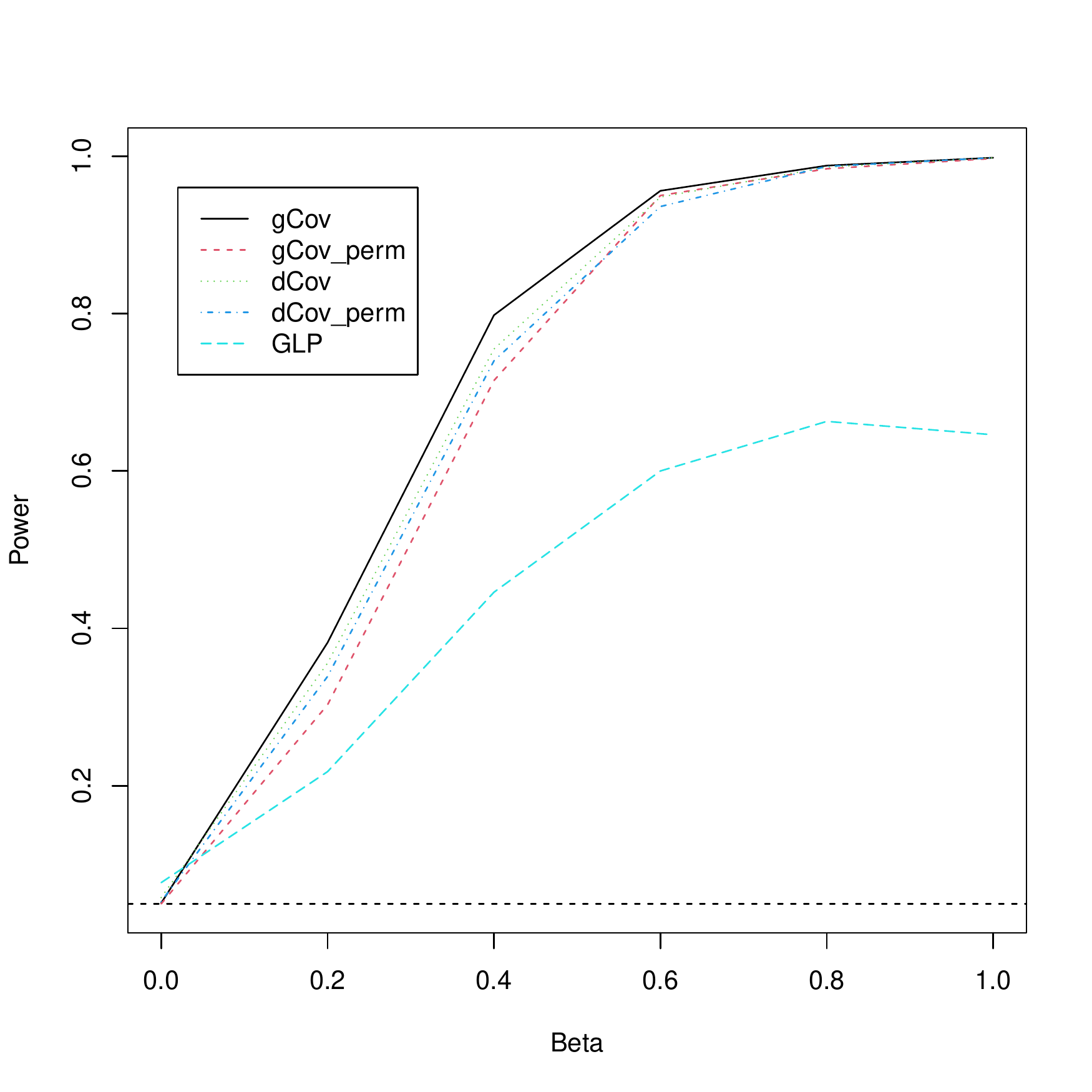} &\hspace{-0.1in}
\includegraphics[width=2.35in,height=2.35in]{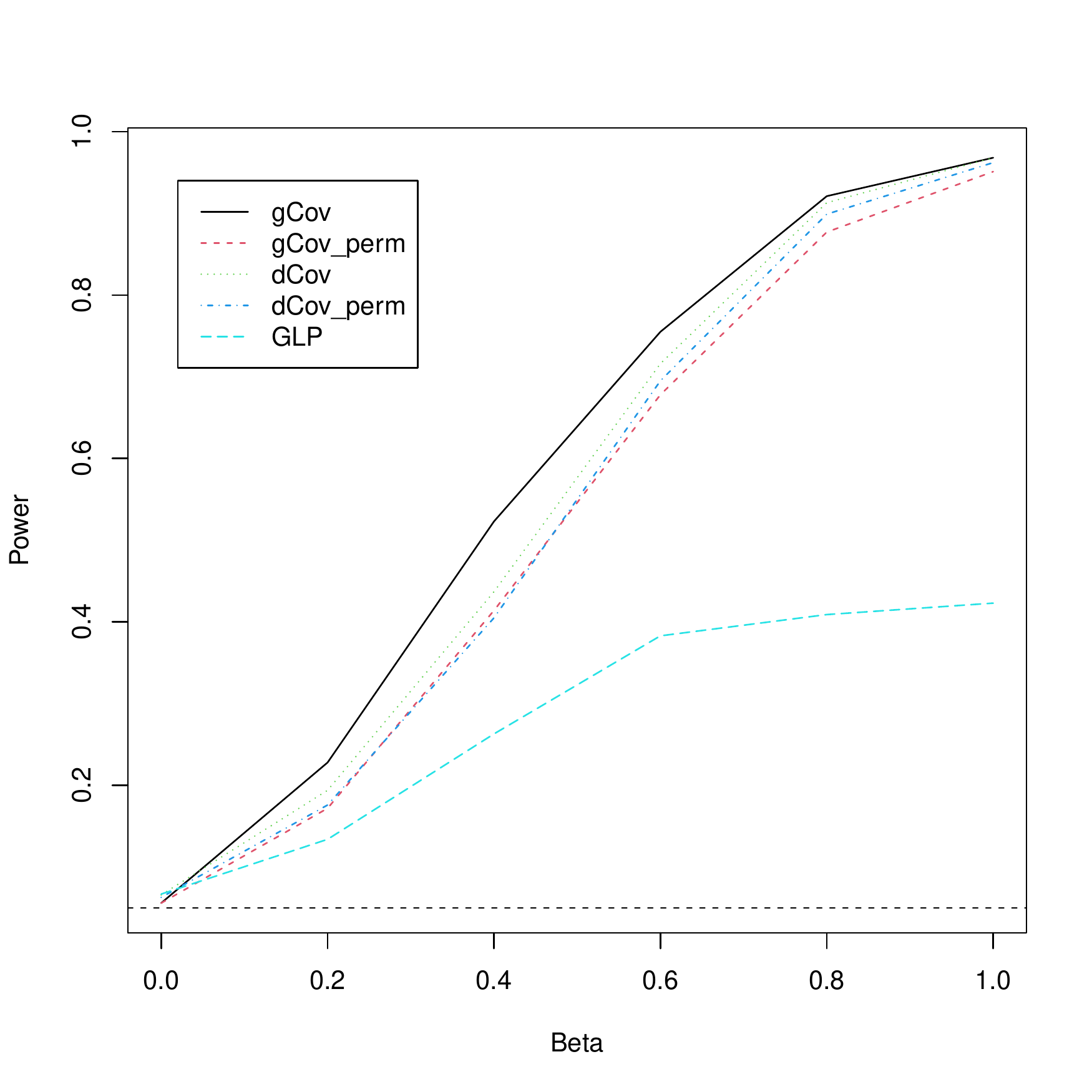}\\ 
(c) $p=500, \bi n=(50,40,30)$ &\hspace{-0.1in} (d) $p=500, \bi n=(72,36,12)$
\end{tabular} \vspace{-0.1in}
\caption{Size and power of tests in Example 3. Dashed horizontal line is the nominal level 0.05. }
\label{fig:ResExp}
\end{figure*}

 \begin{description}
 \item [Example 3.] Let $\bi Z_k = (Z_{k1}, Z_{k2},...Z_{kp})^T -\bi 1_p$, where for $k=1,2,3$ and $j=1,...,p$, $Z_{kj}$'s are i.i.d. from Exp(1). Then generate $\bi X_1 \sim \bi \Sigma_1^{1/2} \bi Z_1$, $\bi X_2 \sim \bi \Sigma_2^{1/2} \bi Z_2$, $\bi X_3 \sim \bi \Sigma_3^{1/2} \bi Z_3$  samples. 
  \end{description}
Although the distributions are not elliptically symmetric, the patterns and observations from this simulation are very similar to those in Example 2 for all tests but GLP. We present the results in Figure \ref{fig:ResExp}.  GLP seems sensitive to the asymmetry of distributions not only in terms of performance as well as in terms of computation.    The GLP algorithm includes a middle step to perform $K$-mean clustering, and that step occasionally stops especially for unbalanced sample sizes. The GLP is slightly oversized and its power is  extremely low.

\section{Real data analysis}\label{sec:realdata}
Two data sets from UCI machine learning repository \cite{Dua19} are studied for $K$ sample tests.

\subsection{LSVT voice rehabilitation data}
 The first data is LSVT Voice Rehabilitation dataset. After speech rehabilitation treatments in Parkinson's disease, 126 patients were evaluated  based on 310 attributes.  Refer to \cite{Tsanas14} for details of the data set and dysphonia measure attributes. Phonations of 42 patients were evaluated  as `acceptable', while 84 patients had `unacceptable' phonations. This data set has the dimension larger than the sample size.  Our goal is to test whether or not phonation features have a same distribution in the `acceptable' group and the `unacceptable' group, which is a $K=2$ sample problem. Before we preform the test, we do some exploratory data analysis to visualize the data in the original high dimensional space and the data projected in low dimensional space.

 \begin{figure*}[thb]
\centering
\begin{tabular}{cc}\vspace{-0.1cm}
\includegraphics[width=2.35in,height=2.35in]{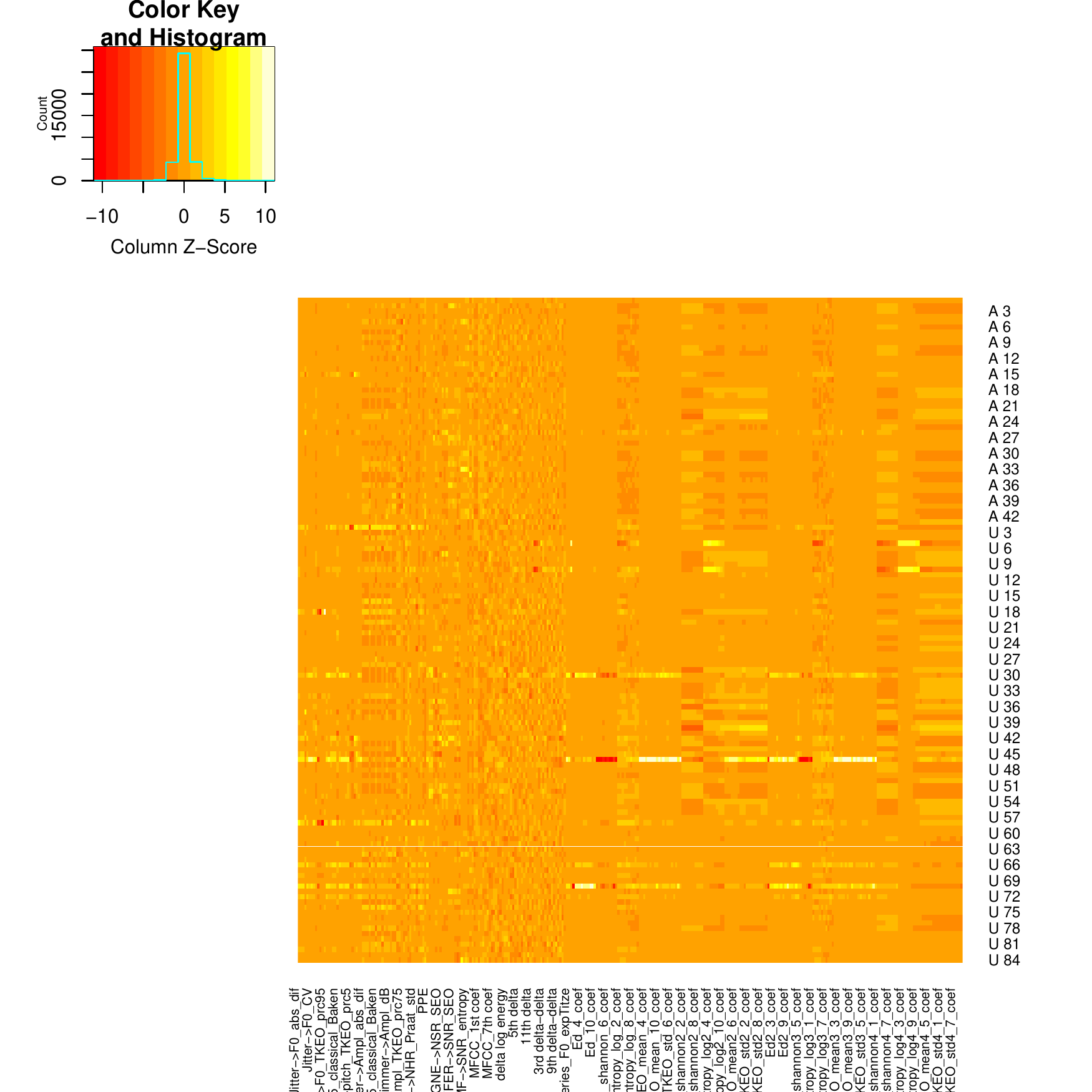} &
\includegraphics[width=2.35in,height=2.35in]{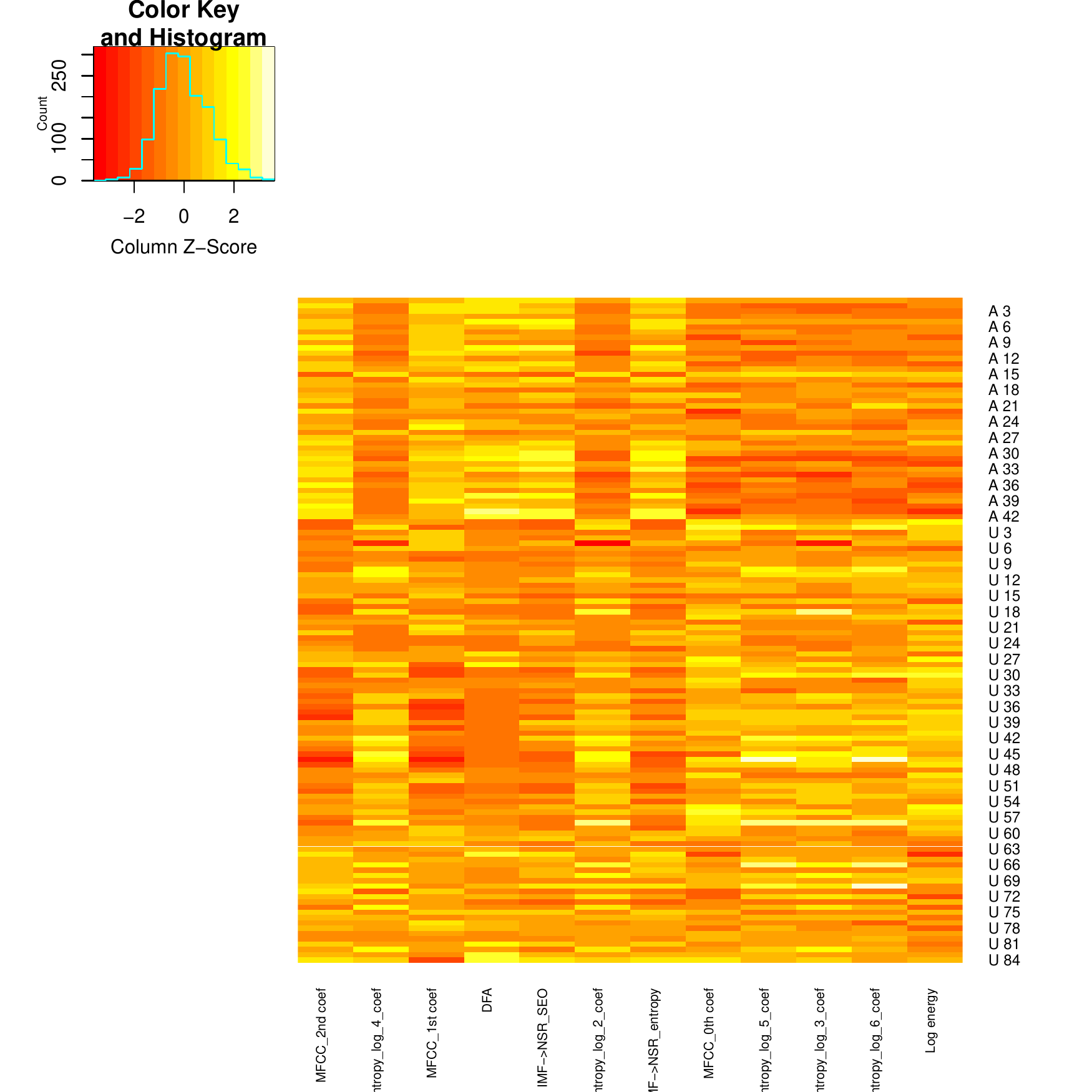}\\
(a) heatmap on all 310 variables  & (b) heatmap on the selected 12 variables \\ \vspace{-0.1cm}
\includegraphics[width=2.35in,height=2.35in]{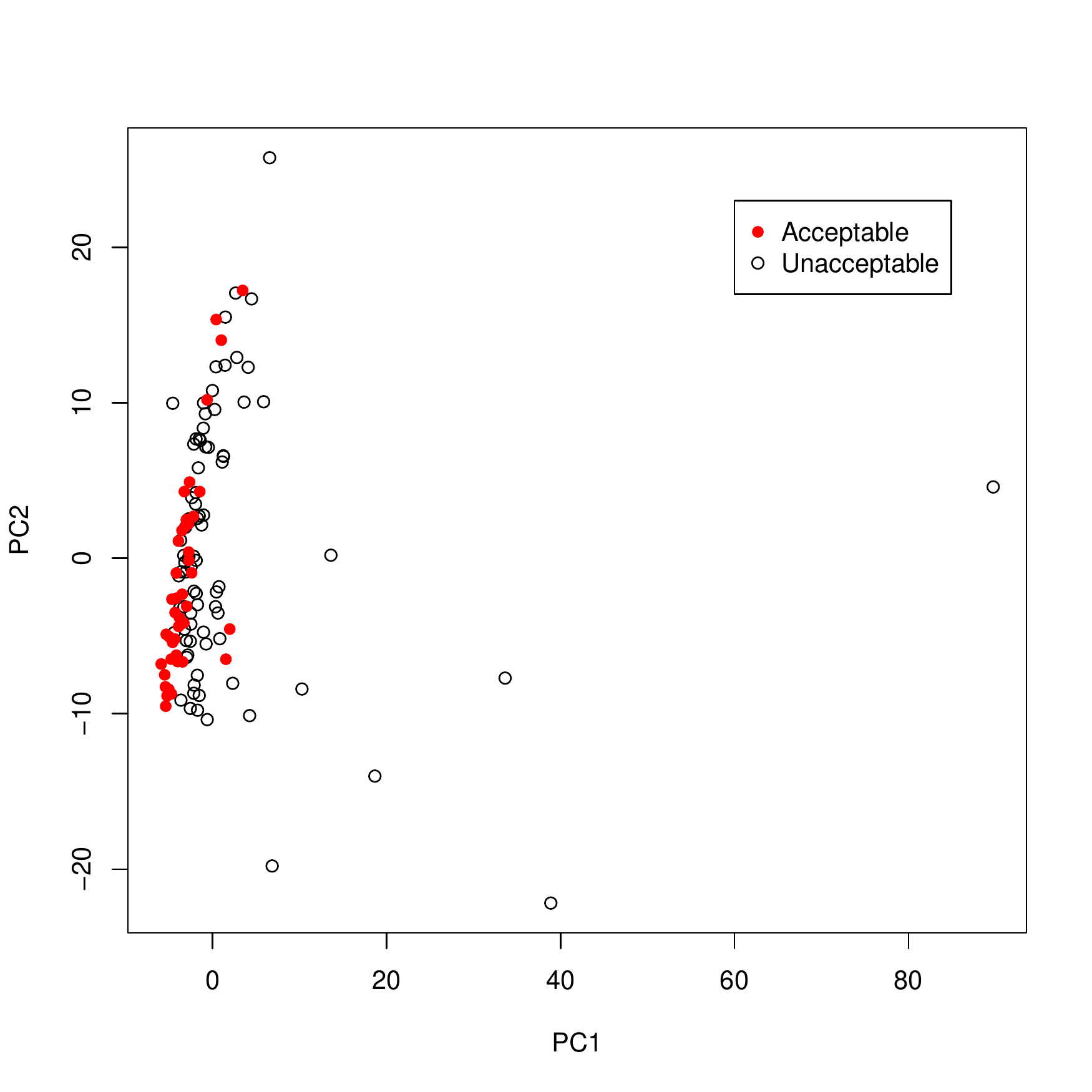} &
\includegraphics[width=2.35in,height=2.35in]{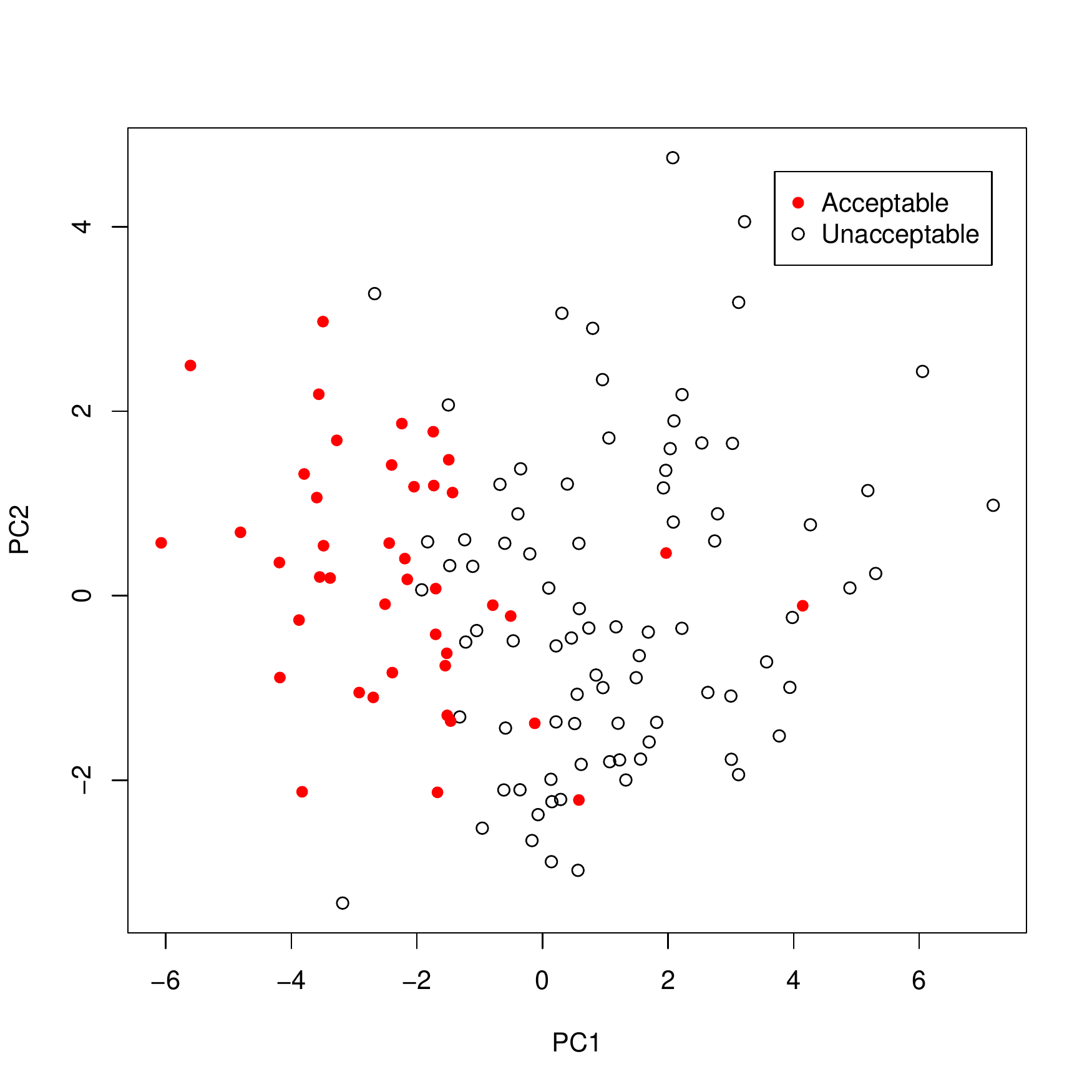}\\ \vspace{-0.1cm}
(c) PCA on all 310 variables  & (d) PCA on the selected 12 variables 
\end{tabular}
\caption{Heatmaps and 2 dimensional PCA projections of Voice rehabilitation data of all 310 variables and  of the selected 12 variables.}
\label{fig:PCA}
\end{figure*}

 A heatmap on all 310 variables is plotted in Figure \ref{fig:PCA}(a) in which the values are centered and scaled by each column variable.  The top third rows are for the acceptable group, while the bottom two thirds for the unacceptable group. It is quite difficult to view differences between two groups. However, the difference shows in the heatmap on the selected 12 variables in Figure \ref{fig:PCA}(b). The selected 12 variables are those with its categorical Gini correlation greater than 0.1. 
 
We also conduct principal component analysis  (PCA) on all variables. The proportions of variance of first two principal components (PC) are 32.29\% and 19.87\%, altogether accounting for 52.16\% of the total variance.  The data are projected on the plane of the first two PC's shown in he left panel of Figure \ref{fig:PCA}(c) in which several patients with unacceptable evaluation are clearly outliers. We also plot data projection on the first two PC's when PCA is conducted on the selected 12 variables in Figure \ref{fig:PCA}(d).  From it, we can see that the unacceptable group tends to have larger values in the first PC.  After a simple feature selection to reduce dimensionality, the separation of two groups is more  evident. In the next, we perform formal tests on equality of distributions of two groups. The test of distributions on all 310 variables and the test of distributions on the 12 selected variables are conducted.

Besides the five methods considered in Section \ref{subsec: test}, five 2-sample test methods are added for comparison. Three methods are proposed in \cite{Li2018} and denoted as Li-loc, Li-scal and Li-both. Sz\'{e}kely's energy test statistic in high dimension is also studied in \cite{Li2018}. It is asymptotically normally distributed,  equivalent to gCov and dCov, but its variance estimation is different and quite complicate in \cite{Li2018} and we include it for comparing its efficiency on variance estimation.   The last considered method denoted as BG is proposed by  Biswas and Ghosh in \cite{Biswas2014}. 
The p-values of those ten methods are reported in Table \ref{tab:voice}. 

\begin{table*}[thb]
\centering
\begin{tabular}{ c |c c c c c } \hline\hline
&gCov &gCov-perm &dCov& dCov-perm & GLP \\ 
 all 310 variables&0.0011& 0.0211& 0.0013 &0.0193& 0.5124\\ 
 12 selected variables & 0.0000& 0.0000&0.0000&0.0000&0.0000\\ \hline 
&Li-loc &Li-scal &Li-both &Szekely & BG \\   
 all 310 variables &0.0204& 0.3190 &0.0197& 0.0166 &0.3272\\
 12 selected variables & 0.0000&0.0958&0.0000&0.0000&0.0009\\ \hline \hline
\end{tabular}
\caption{p-values of various 2 sample tests for all features and for the select 12 features in LSVT Voice rehabilitation data.}
\label{tab:voice}
\end{table*}

With the feature selection to reduce dimension,  all methods except for Li-scal strongly reject the equality of two distributions. While for the high dimension data,  three methods GLP, Li-scal and BG fail to conclude different distributions in two groups. The gCov and dCov methods provide the most significant evidence on the differences of two groups. 
 
 \subsection{Arcene data} 
 The second data set we apply to is Arcene mass-spectrometric data for 900 patients from cancer group and healthy group. The data set was merged from three resources on ovarian cancer data and prostate cancer data.  The preprocessing steps of limiting the mass range, averaging the technical repeats, removing the baseline, smoothing, rescaling and aligning the spectra were prepared to reduce disparity between data sources. Arcene data have 10000 features including 7000 real features and 3000 random probes. The dimension is much higher than the sample size. The data was formatted for benchmarking variable selection algorithms for the two class classification problem in 2003 NIPS, the top conference on machine mining and computational neuroscience.  The data were partitioned to training, validation, and test sets. For the training and validation sets, each has 44 cancer positives and 56 negatives, while the test set has 310 positives and 390 negatives. Refer to \cite{Guyon04} for details about the data preparation and NIPS challenge results. 
  \begin{table}[thb]
\centering
\begin{tabular}{ c |c c c c c } \hline\hline
&gCov &gCov-perm &dCov& dCov-perm & GLP \\ \hline
p-value&  0.5394& 0.4530& 0.4132& 0.3160& 0.0389 \\
\hline\hline
\end{tabular}
\caption{p-values of testing whether training data, testing data and validation data in ARCENE have a same distribution.}
\label{tab:arcene}
\end{table}

Rather than conducting two sample test, we perform 3 sample testing on distributional equality of  the training data, validation data and test data. That is the assumption and the logic behind the procedure of using training data to build model, using validation data to select model  and using test data to assess model.  P-values of five methods are reported in Table \ref{tab:arcene}. Only GLP rejects the equality with p-value 0.0389, while the other four methods with large p-values support the distribution equality assumption that makes the data mining challenge competition valid.

\section{Conclusions and future work}\label{sec:conclusion}
The categorical Gini correlation is an alternative to the distance correlation to measure the correlation  between a $p$-variate numeric variable $\bi X$ and a categorical variable $Y$.  But the Gini one has more appealing properties such as nice presentation and better interpretation. When $p$ is fixed, Dang $et$ $al.$ \cite{Dang2021} showed that the sample Gini correlation converges in distribution to a quadratic form of normal distributions under independence of $\bi X$ and $Y$. 
In this paper, we have studied the inference of the categorical Gini correlation in a more realistic setting  where both the sample size and the dimensionality are diverging in an arbitrary fashion.  One of our main results, Theorem \ref{gcov1}, reveals that those complicated quadratic forms of normal random variables admit a normal limit as the dimensionality $p$ diverges to infinity, providing an intriguing example to understand the distinction between classical and high-dimensional theory. 

Based on these asymptotic distributions, a new consistent $K$-sample test has been developed. Both simulation studies and real data illustrations have shown the proposed test performs uniformly better than the distance correlation based test for unbalanced cases. 

The Gini covariance has been generalized to a reproducing kernel Hilbert space (RKHS)  in \cite{Zhang21} as follows.
\begin{align} \label{kernelgcov}
\mbox{gCov}(\bi X,Y; d_\kappa) &=\E  \{ d_\kappa(\bi X_1,\bi X_2)\}-\sum_{k=1}^Kp_k\E \{d_\kappa(\bi X_1^{(k)},\bi X_2^{(k)})\},
\end{align}
where $d_\kappa(\bi x_1, \bi x_2) = \sqrt{\kappa(\bi x_1, \bi x_1)+\kappa(\bi x_2, \bi x_2)-2\kappa(\bi x_1, \bi x_2)}$, the distance in the feature space induced by positive definite kernel $\kappa$. More specifically, a positive definite kernel, $\kappa: \mathbb{R}^p \times \mathbb{R}^p \rightarrow \mathbb{R}$, implicitly defines an embedding map: 
\begin{align*}
\phi: \bi x \in \mathbb{R}^p \mapsto \phi(\bi x) \in {\cal F},
\end{align*}
via an inner product in the feature space ${\cal F}$:
\begin{align*}
\kappa(\bi x_1, \bi x_2) = \langle \phi(\bi x_1), \phi(\bi x_2) \rangle, \;\;\; \bi x_1, \bi x_2 \in \mathbb{R}^p. 
\end{align*} 
Replacing the expectations in (\ref{kernelgcov}) by the corresponding $U$-statistics and replacing $p_k$ by $\hat{p}_k$, we obtain the sample kernelized Gini covariance $\mbox{gCov}_n(\bi X, Y; d_\kappa)$.  With a choice of bounded kernel such as the popular radial basis function kernel (RBF), the moment condition \textbf{C}1 can be dropped. It will be interesting to derive similar results for kernel covariance and correlation.  

As long as pairwise (dis)similarities are available, kernel Gini covariance can be used for complex data type.  It is interesting to adopt kernel Gini covariance and correlation based on neural tangent kernel (NTK) in the study of deep artificial neural networks (ANN).  Continuations of this work could take those directions as well as the following. 
\begin{itemize}
\item The permutation test based on Gini covariances in high dimension has demonstrated its size and power empirically. A theoretical and rigorous treatment is needed. 
\item  When $\bi X$ and $Y$ are dependent, the CLT holds for the sample Gini covariance $\mbox{gCov}_n$. 
Under the null that $\bi X$ and $Y$ are independent,  $\mbox{gCov}_n$ is a $U$-statistic representation with first order degeneracy but admits a normal limit in the high dimension. Therefore,  we would expect a non-null CLT for $\mbox{gCov}_n$ when $p \to \infty$.
\item In this study, the number of levels of $Y$ is fixed and finite. However, some applications like Poisson process have infinity levels.  In some applications like discretization procedure, the number of levels might increase as sample size increases.   It is interesting to study estimation of Gini correlation in those cases and explore its asymptotical distribution when $n$, $p$ and $K$ diverge. 
\end{itemize}
\section{Appendix}
\noindent
Let $\bi X, \bi X_1, \bi X_2$, $\bi X_3$ and $\bi X_4$ be independent random variables from $F$.
We will adopt the following notations through this section. 
\begin{align*}
&\xi(\bi X_1)=\E \Big(d^2(\bi X, \bi X_1)\big|\bi X_1\Big),\\  
&\sigma^2=\E \xi (\bi X_1),\\
&\gamma^4=\E \Big(\xi^2 (\bi X_1)\Big),\\
&\eta (\bi X_1, \bi X_2)=\E \Big(\big(d(\bi X, \bi X_1)\bigm|\bi X_1\big)\big(d(\bi X, \bi X_2)\bigm|\bi X_2\big)\Big),\\
&\tau^4=\E\big(\eta(\bi X_1, \bi X_2)\big)^2,\\
&\omega^4=\E d^4(\bi X_1, \bi X_2).
\end{align*}
It is easy to check that  $\gamma^4>\sigma^4>\tau^4$ and $\omega^4 >\sigma^4$ by  Jensen's inequality. 

\subsection{Lemmas} 
Before we prove the major result in Theorem \ref{gcov1}, let us provide several necessary lemmas and their proofs. The remaining lemmas shall be given in the proof of Theorem \ref{gcov1}. 
The double centered distance $d(\cdot, \cdot)$ in (\ref{d}) has appealing orthogonal properties  in the following Lemmas \ref{prop_d} and \ref{prop_eta}.

\begin{lemma}\label{prop_d}
If $\E\|\bi X\|^4 <\infty$, then
$d(\cdot, \cdot)$ in (\ref{d}) satisfies 
\begin{enumerate}
\item $\E d(\bi X_1, \bi X_2) =0$;
\item $\E \big(d(\bi X_1, \bi X_2)\bigm|\bi X_1\big) = \E\big(d(\bi X_1,\bi X_2)\bigm|\bi X_2\big) =0$;
\item $\E\big( d(\bi X_1, \bi X_2)d(\bi X_1, \bi X_3)\big) =0$;
\item $\E\big(d(\bi X_1, \bi X)d(\bi X_2, \bi X) d(\bi X_3, \bi X)d(\bi X_4, \bi X)\big)=0$;
\item $\E\big(d^2(\bi X_1, \bi X)d(\bi X_2, \bi X)d(\bi X_3, \bi X)\big)=0$;
\item $\E \big(d^3(\bi X_1, \bi X)d(\bi X_2, \bi X)\big)=0$.
\item $\E\big(d^2(\bi X_1, \bi X_2)\big)=\sigma^2$;
\item $\E\big(d^2(\bi X_1, \bi X_2)d^2(\bi X_1, \bi X_3)\big)=\gamma^4$.
\end{enumerate}
\end{lemma}
\begin{proof}
It is straightforward to obtain that  $\E d(\bi X_1, \bi X_2)=0$ and 
$$\E \big(d(\bi X_1, \bi X_2)\bigm|\bi X_1\big) = \E\big(d(\bi X_1,\bi X_2)\bigm|\bi X_2\big) =0.$$
By the double expectation argument, we have 
\begin{align*}
 \E\big(d(\bi X_1, \bi X_2)d(\bi X_1, \bi X_3)\big)&= \E\left\{ \E\big( d(\bi X_1, \bi X_2) d(\bi X_1, \bi X_3)\bigm|\bi X_1\big)\right \}\\
& =\E\left\{ \E\big(d(\bi X_1, \bi X_2)\bigm|\bi X_1\big)\E \big(d(\bi X_1, \bi X_3)\bigm|\bi X_1\big) \right\}\\
&= 0.
\end{align*}
The other properties can be proved similarly. 
\end{proof}

\begin{lemma}\label{prop_eta}
If $\E\|\bi X\|^4 <\infty$, we have
 \begin{enumerate}
\item $\E \eta(\bi X_1, \bi X_2)=0,$
\item $\E\big(\eta(\bi X_1, \bi X_2)\eta(\bi X_1, \bi X_3)\big)=0,$
\item$\E[d(\bi X_1, \bi X_3)d(\bi X_2, \bi X_3)d(\bi X_1, \bi X_4)d(\bi X_2, \bi X_4)]=\tau^4,$
\item $\E\big(\xi(\bi X_1)\eta(\bi X_1, \bi X_2)\big)^2=0.$
\end{enumerate}
\end{lemma}
\begin{proof} $\E \eta(\bi X_1, \bi X_2)=0$ follows directly from the property 3 in Lemma \ref{prop_d}. 
Using the double expectation argument and properties in Lemma \ref{prop_d} , we have 
 \begin{align*}
&\E\big(\eta(\bi X_1, \bi X_2)\eta(\bi X_1, \bi X_3)\big)\\
&=\E\left \{\E\big(d(\bi X, \bi X_1)d(\bi X, \bi X_2)\bigm|\bi X_1, \bi X_2\big) \E\big(d(\bi X', \bi X_1)d(\bi X', \bi X_3)\bigm|\bi X_1, \bi X_3\big)  \right \}\\
&=\E\big(d(\bi X, \bi X_1)d(\bi X, \bi X_2)d(\bi X', \bi X_1)d(\bi X', \bi X_3)\big)\\
&=\E\left \{\E\big(d(\bi X, \bi X_1)d(\bi X, \bi X_2)d(\bi X', \bi X_1)d(\bi X', \bi X_3)\bigm|\bi X, \bi X'\big)  \right \}\\
&=\E\left \{\E\big(d(\bi X, \bi X_1)d(\bi X', \bi X_1)\bigm|\bi X, \bi X'\big) \E\big(d(\bi X, \bi X_2)\bigm|\bi X\big) \E\big(d(\bi X', \bi X_3)\bigm|\bi X'\big)  \right \} =0,\\
&\E[\eta(\bi X_1, \bi X_2)]^2\\
&=\E\left \{\E \big(d(\bi X, \bi X_1)d(\bi X, \bi X_2)|\bi X_1, \bi X_2\big) \E\big(d(\bi X', \bi X_1)d(\bi X', \bi X_2)|\bi X_1, \bi X_2\big)  \right \}\\
&=\E\big(d(\bi X_1, \bi X_3)d(\bi X_2, \bi X_3)d(\bi X_1, \bi X_4)d(\bi X_2, \bi X_4)\big)= \tau^4,\\
&\E\big(\xi(\bi X_1)\eta(\bi X_1, \bi X_2)\big)^2\\
&=\E\left \{\E\big(d^2(\bi X, \bi X_1)\bigm|\bi X_1\big) \E\big(d(\bi X', \bi X_1)\bigm|\bi X_1\big) \E\big(d(\bi X'', \bi X_2)\bigm|\bi X_2\big)  \right \}\\
&=\E\big(d^2(\bi X, \bi X_1)d(\bi X', \bi X_1)d(\bi X'', \bi X_2)\big)\\
&=\E\left\{ \E\big(d^2(\bi X, \bi X_1)d(\bi X', \bi X_1)\bigm|\bi X, \bi X'\big) \E\big(d(\bi X_2, \bi X'')\bigm|\bi X''\big)\right \} =0.
\end{align*}
This completes the proof of Lemma \ref{prop_eta}.
\end{proof}

 \begin{lemma} \label{ncond}
 Under conditions \textbf{C2}, 
 \begin{align*}
 \dfrac{\gamma^4}{n\sigma^4} \to 0.
 \end{align*}
 \end{lemma}
 \begin{proof} By the Cauchy-Schwarz  inequality, it is easy to obtain that 
 \begin{align*}
 \gamma^4&=\E\big(d^2(\bi X_1, \bi X)d^2(\bi X_2, \bi X)\big)\\
 & \leq  \big(\E d^4(\bi X_1, \bi X)\big)^{1/2} \big(\E d^4(\bi X_2, \bi X)\big)^{1/2}\\
 &=\E d^4(\bi X_1, \bi X_2).
 \end{align*}
 By condition \textbf{C}2, we have $\dfrac{\gamma^4}{n\sigma^4} \to 0.$
 \end{proof}
\noindent
\subsection{Proof of Theorem \ref{gcov1}}
\noindent
Under  independence of $\bi X$ and $Y$, by Lemma \ref{prop_d}, we have  
\begin{align}
\sigma^2_0&=Var\bigg({n \choose 2}^{-1}\sum_{1 \leq i<j \leq n}d(\bi X_i, \bi X_j)\bigg)+\sum_{k=1}^K \hat{p}^2_k Var\bigg({n_k \choose 2}^{-1}\sum_{1 \leq i<j \leq n_k}d(\bi X^{(k)}_i,\bi X^{(k)}_j)\bigg) \nonumber\\
&-2{n \choose 2}^{-1}\sum_{k=1}^K \hat{p}_k {n_k \choose 2}^{-1}Cov\bigg(\sum_{1 \leq i <j \leq n}d(\bi X_i, \bi X_j), \sum_{1 \leq i<j \leq n_k}d(\bi X^{(k)}_i,\bi X^{(k)}_j)\bigg) \nonumber \\
&={n \choose 2}^{-1}Var\bigg(d(\bi X_1, \bi X_2)\bigg)+\sum_{k=1}^K \hat{p}^2_k {n_k \choose 2}^{-1} Var\bigg(d(\bi X^{(k)}_1,\bi X^{(k)}_2)\bigg) \nonumber \\
&-2{n \choose 2}^{-1}\sum_{k=1}^K \hat{p}_k Var\bigg(d(\bi X^{(k)}_1,\bi X^{(k)}_2)\bigg) \nonumber\\
&=\bigg(\sum_{k=1}^K \hat{p}^2_k {n_k \choose 2}^{-1}-{n \choose 2}^{-1}\bigg) \E d^2(\bi X_1, \bi X_2) \label{eqn:vargcov} \\
&=\big(\dfrac{2K-2}{n^2}+o(n^{-2})\big)\E d^2(\bi X_1, \bi X_2), \label{eqn:vargcov1}
\end{align}
where 
$\E d^2(\bi X_1, \bi X_2) = V^2(\bi X)$ is the squared distance variance of $\bi X$ in \cite{Szekely2007}.

For a short presentation, we denote $gCov_n(\bi X, Y)$ as $G_n$, which is 
\begin{align*}
G_n
&:={n \choose 2}^{-1}\sum_{1 \leq i<j \leq n}
d(\bi X_i, \bi X_j)-\sum_{k=1}^K \hat{p}_k{n_k \choose 2}^{-1}\sum_{1 \leq i<j \leq n_k}d(\bi X^{(k)}_i,\bi X^{(k)}_j).
\end{align*}
In order to show the asymptotic normality of $G_{n}$, we construct a martingale sequence as follows. Assume that $\bi X_i$'s have been sorted by $Y_i$'s,  that is, $\bi X_i=\bi X^{(1)}_i$, for  $i=1, 2, ..., n_1$;  $\bi X_{n_1+i}=\bi X^{(2)}_i$, for $i=1, ..., n_2$;  ...;  $\bi X_{n_1+...+n_{k-1}+i}=\bi X^{(k)}_i$, for $i=1, ..., n_k$.
 Let $\mathscr{F}_0=\{\emptyset, \Omega\}$,  $\mathscr{F}_l=\sigma\{\bi X_1, ..., \bi X_l\}$ with $l=1,2,...,n$. $\E_l$ denotes the conditional expectation given $\mathscr{F}_l$. Define 
\begin{align*}
M_{n, l}=(\E_l-\E_{l-1})G_{n}.
\end{align*}
$\{M_{n, l}, 1 \leq l \leq n\}$ is a martingale difference sequence with respect to the nested $\sigma$-fields $\{\mathscr{F}_l, 1 \leq l\leq n\}$. Also under the independence, 
\begin{align*}
\sum^{n}_{l=1}M_{n, l}=(\E_{n}-\E_0)G_{n}=G_{n}-\E G_{n} = G_n. 
\end{align*}
We need to establish the asymptotic normality of $\sum_{l=1}^n M_{n,l}$. Without loss of generality, we will prove the case for $K=3$.

We first work out the representations of $M_{n,l}$ by using the properties in Lemmas \ref{prop_d} and \ref{prop_eta}.  Depending on $l$, $M_{n,l}$ have three forms. 

\noindent{\underline{\bf{Case 1}}}, for $1\leq l \leq n_1$, $\mathscr{F}_l=\sigma\{\bi X^{(1)}_1, ..., \bi X^{(1)}_l\}$.   We have 
\begin{align*}
&\E(G_{n}| \mathscr{F}_l)\\
&={n \choose 2}^{-1}\E\bigg( \sum_{1 \leq i <j \leq n}d(\bi X_i, \bi X_j)|\mathscr{F}_{l}\bigg)-{n_1 \choose 2}^{-1}\hat{p}_1 \E\bigg( \sum_{1 \leq i <j \leq n_1}d(\bi X^{(1)}_i, \bi X^{(1)}_j)|\mathscr{F}_{l}\bigg)\\
&={n \choose 2}^{-1}\sum_{1 \leq i <j \leq l}d(\bi X_i, \bi X_j)-\hat{p}_1 {n_1 \choose 2}^{-1}\sum_{1 \leq i <j \leq l}d(\bi X^{(1)}_i, \bi X^{(1)}_j)\\
&=-\dfrac{2(n-n_1)}{n(n-1)(n_1-1)}\sum_{1 \leq i <j \leq l}d(\bi X^{(1)}_i, \bi X^{(1)}_j),\\
&\E(G_{n}| \mathscr{F}_{l-1})=
-\dfrac{2(n-n_1)}{n(n-1)(n_1-1)}\sum_{1 \leq i <j \leq l-1}d(\bi X^{(1)}_i, \bi X^{(1)}_j).
\end{align*}
Thus,
\begin{align*}
M_{n, l}=(\E_{l}-\E_{l-1})G_{n}=-\dfrac{2(n-n_1)}{n(n-1)(n_1-1)}\sum^{l-1}_{j=1}d(\bi X_l, \bi X^{(1)}_j).
\end{align*}
\noindent{\underline{\bf{Case 2}}}, for $n_1<l \leq n_1+n_2$, $\mathscr{F}_l=\sigma\{\bi X^{(1)}_1, ...,\bi X^{(1)}_{n_1}, \bi X^{(2)}_1..., \bi X^{(2)}_{l-n_1}\}$. We have 
\begin{align*}
&\E(G_{n}| \mathscr{F}_l)\\
&={n \choose 2}^{-1}\E\bigg( \sum_{1 \leq i <j \leq n}d(\bi X_i, \bi X_j)|\mathscr{F}_{l}\bigg)-{n_1 \choose 2}^{-1}\hat{p}_1 \E\bigg( \sum_{1 \leq i <j \leq n_1}d(\bi X^{(1)}_i, \bi X^{(1)}_j)|\mathscr{F}_{l}\bigg)\\
&-{n_2 \choose 2}^{-1}\hat{p}_2 \E\bigg( \sum_{1 \leq i <j \leq n_2}d(\bi X^{(2)}_i, \bi X^{(2)}_j)|\mathscr{F}_{l}\bigg)\\
&={n \choose 2}^{-1}\left\{\sum_{1 \leq i<j \leq n_1}d(\bi X^{(1)}_i, \bi X^{(1)}_j)+\sum_{1\leq i<j \leq l-n_1}d(\bi X^{(2)}_i, \bi X^{(2)}_j)+\sum_{j=1}^{l-n_1}\sum_{i=1}^{n_1}d(\bi X^{(1)}_i, \bi X^{(2)}_j) \right\}\\
&-\hat{p}_1 {n_1 \choose 2}^{-1}\sum_{1 \leq i <j \leq n_1}d(\bi X^{(1)}_i, \bi X^{(1)}_j)-\hat{p}_2{n_2 \choose 2}^{-1} \sum_{1 \leq i <j \leq l-n_1}d(\bi X^{(2)}_i, \bi X^{(2)}_j)\\
&=\bigg({n \choose 2}^{-1}-\hat{p}_1 {n_1 \choose 2}^{-1}\bigg)\sum_{1 \leq i<j \leq n_1}d(\bi X^{(1)}_i, \bi X^{(1)}_j)\\
&+\bigg({n \choose 2}^{-1}-\hat{p}_2 {n_2 \choose 2}^{-1}\bigg)\sum_{1 \leq i<j \leq l-n_1}d(\bi X^{(2)}_i, \bi X^{(2)}_j)+{n \choose 2}^{-1}\sum_{j=1}^{l-n_1}\sum_{i=1}^{n_1}d(\bi X^{(1)}_i, \bi X^{(2)}_j) \\
&=-\dfrac{2(n-n_1)}{n(n-1)(n_1-1)}\sum_{1 \leq i <j \leq n_1}d(\bi X^{(1)}_i, \bi X^{(1)}_j)-\dfrac{2(n-n_2)}{n(n-1)(n_2-1)}\sum_{1 \leq i <j \leq l-n_1}d(\bi X^{(2)}_i, \bi X^{(2)}_j)\\
&+{n \choose 2}^{-1}\sum_{j=1}^{l-n_1}\sum_{i=1}^{n_1}d(\bi X^{(1)}_i, \bi X^{(2)}_j). 
\end{align*}
Therefore,
\begin{align*}
M_{n, l}= -\dfrac{2(n-n_2)}{n(n-1)(n_2-1)}\sum^{l-n_1-1}_{j=1}d(\bi X_l, \bi X^{(2)}_j)+{n \choose 2}^{-1}\sum_{i=1}^{n_1}d(\bi X_l, \bi X^{(1)}_i).
\end{align*}
\noindent{\underline{\bf{Case 3}}}, for  $n_1+n_2<l \leq n$, $\mathscr{F}_l=\sigma\{\bi X^{(1)}_1, ...,\bi X^{(1)}_{n_1}, \bi X^{(2)}_1...,\bi X^{(2)}_{n_2}, \bi X^{(3)}_1,..., \bi X^{(3)}_{l-n_1-n_2}\}$. We have 
\begin{align*}
&\E(G_{n}| \mathscr{F}_l)\\
&={n \choose 2}^{-1}\E\bigg( \sum_{1 \leq i <j \leq n}d(\bi X_i, \bi X_j)|\mathscr{F}_{l}\bigg)-{n_1 \choose 2}^{-1}\hat{p}_1 \E\bigg( \sum_{1 \leq i <j \leq n_1}d(\bi X^{(1)}_i, \bi X^{(1)}_j)|\mathscr{F}_{l}\bigg)\\
&-{n_2 \choose 2}^{-1}\hat{p}_2 \E\bigg( \sum_{1 \leq i <j \leq n_2}d(\bi X^{(2)}_i, \bi X^{(2)}_j)|\mathscr{F}_{l}\bigg)-{n_3 \choose 2}^{-1}\hat{p}_3 \E\bigg( \sum_{1 \leq i <j \leq n_3}d(\bi X^{(3)}_i, \bi X^{(3)}_j)|\mathscr{F}_{l}\bigg)\\
&={n \choose 2}^{-1}\left\{\sum_{1 \leq i<j \leq n_1}d(\bi X^{(1)}_i, \bi X^{(1)}_j)+\sum_{1\leq i<j \leq n_2}d(\bi X^{(2)}_i, \bi X^{(2)}_j)+\sum_{1\leq i<j \leq l-n_1-n_2}d(\bi X^{(3)}_i, \bi X^{(3)}_j)\right. \\
&\left. +\sum_{i=1}^{n_1}\sum_{j=1}^{n_2}d(\bi X^{(1)}_i, \bi X^{(2)}_j) +\sum_{i=1}^{n_1}\sum_{j=1}^{l-n_1-n_2}d(\bi X^{(1)}_i, \bi X^{(3)}_j)+\sum_{i=1}^{n_2}\sum_{j=1}^{l-n_1-n_2}d(\bi X^{(2)}_i, \bi X^{(3)}_j)\right\}\\
&-\hat{p}_1 {n_1 \choose 2}^{-1}\sum_{1 \leq i <j \leq n_1}d(\bi X^{(1)}_i, \bi X^{(1)}_j)-\hat{p}_2{n_2 \choose 2}^{-1} \sum_{1 \leq i <j \leq n_2}d(\bi X^{(2)}_i, \bi X^{(2)}_j)\\
&-\hat{p}_3{n_3 \choose 2}^{-1} \sum_{1 \leq i <j \leq l-n_1-n_2}d(\bi X^{(3)}_i, \bi X^{(3)}_j)\\
&=\bigg({n \choose 2}^{-1}-\hat{p}_1 {n_1 \choose 2}^{-1}\bigg)\sum_{1 \leq i<j \leq n_1}d(\bi X^{(1)}_i, \bi X^{(1)}_j)
+\bigg({n \choose 2}^{-1}-\hat{p}_2 {n_2 \choose 2}^{-1}\bigg)\sum_{1 \leq i<j \leq n_2}d(\bi X^{(2)}_i, \bi X^{(2)}_j)\\
&+\bigg({n \choose 2}^{-1}-\hat{p}_3 {n_3 \choose 2}^{-1}\bigg)\sum_{1 \leq i<j \leq l-n_1-n_2}d(\bi X^{(3)}_i, \bi X^{(3)}_j)+{n \choose 2}^{-1}\sum_{i=1}^{n_1}\sum_{j=1}^{n_2}d(\bi X^{(1)}_i, \bi X^{(2)}_j) \\
&+{n \choose 2}^{-1}\sum_{i=1}^{n_1}\sum_{j=1}^{l-n_1-n_2}d(\bi X^{(1)}_i, \bi X^{(3)}_j) 
+{n \choose 2}^{-1}\sum_{i=1}^{n_2}\sum_{j=1}^{l-n_1-n_2}d(\bi X^{(2)}_i, \bi X^{(3)}_j) \\
&=-\dfrac{2(n-n_1)}{n(n-1)(n_1-1)}\sum_{1 \leq i <j \leq n_1}d(\bi X^{(1)}_i, \bi X^{(1)}_j)-\dfrac{2(n-n_2)}{n(n-1)(n_2-1)}\sum_{1 \leq i <j \leq n_2}d(\bi X^{(2)}_i, \bi X^{(2)}_j)\\
&+\dfrac{2(n-n_3)}{n(n-1)(n_3-1)}\sum_{1 \leq i <j \leq l-n_1-n_2}d(\bi X^{(3)}_i, \bi X^{(3)}_j)
+{n \choose 2}^{-1}\sum_{i=1}^{n_1}\sum_{j=1}^{n_2}d(\bi X^{(1)}_i, \bi X^{(2)}_j) \\
&+{n \choose 2}^{-1}\sum_{i=1}^{n_1}\sum_{j=1}^{l-n_1-n_2}d(\bi X^{(1)}_i, \bi X^{(3)}_j) 
+{n \choose 2}^{-1}\sum_{i=1}^{n_2}\sum_{j=1}^{l-n_1-n_2}d(\bi X^{(2)}_i, \bi X^{(3)}_j). 
\end{align*}
Thus,
\begin{align*}
M_{n, l}&=- \dfrac{2(n-n_3)}{n(n-1)(n_3-1)}\sum^{l-n_1-n_2-1}_{j=1}d(\bi X_l, \bi X^{(3)}_j)+{n \choose 2}^{-1}\sum_{i=1}^{n_1}d(\bi X_l, \bi X^{(1)}_i)\\
&+{n \choose 2}^{-1}\sum_{i=1}^{n_2}d(\bi X_l, \bi X^{(2)}_i).
\end{align*}

In order to apply martingale central limit theorem to the constructed martingale sequence, $M_{n, l}, l=1,...,n$, we need the following Lemma \ref{mct_1} and Lemma \ref{mct_2}.  
\begin{lemma} \label{mct_1}
Under conditions \textbf{C}1-\textbf{C}3 and  independence of $\bi X$ and $Y$, as $\min\{n_1, n_2,..., n_K\} \to \infty$, we have 
\begin{align*}
\dfrac{\sum_{l=1}^n \sigma^2_{n,l}}{\sigma^2_0} \to 1\;\; \mbox{ in probability}, 
\end{align*}
where $\sigma^2_{n,l}=\E_{l-1}(M^2_{n, l}).$
\end{lemma}
\begin{proof} We first obtain three formulas of $\sigma_{n,l}^2$ according to $l$.

\noindent{\underline{\bf{Case 1}}}, for $l \leq n_1$, we have 
\begin{align*}
\sigma^2_{n,l}&=\E_{l-1}(M^2_{n,l})=\E\left \{\Big( -\dfrac{2(n-n_1)}{n(n-1)(n_1-1)}\sum^{k-1}_{j=1}d(\bi X_l, \bi X^{(1)}_j)\Big)^2|\mathscr{F}_{l-1}\right \}\\
&=\dfrac{4(n-n_1)^2}{n^2(n-1)^2(n_1-1)^2}\E\left \{ \sum_{i=1}^{l-1}\sum^{l-1}_{j=1}d(\bi X_l, \bi X^{(1)}_i)d(\bi X_l, \bi X^{(1)}_j)| \mathscr{F}_{l-1} \right \}\\
&=\dfrac{4(n-n_1)^2}{n^2(n-1)^2(n_1-1)^2}\sum_{i=1}^{l-1}\sum^{l-1}_{j=1}\E\left \{d(\bi X_l, \bi X^{(1)}_i)d(\bi X_l, \bi X^{(1)}_j)| \mathscr{F}_{l-1} \right \}\\
&=\dfrac{4(n-n_1)^2}{n^2(n-1)^2(n_1-1)^2} \sum_{i=1}^{l-1}\sum^{l-1}_{j=1}\E\left \{d(\bi X_l, \bi X^{(1)}_i)d(\bi X_l, \bi X^{(1)}_j)| \bi X^{(1)}_i, \bi X^{(1)}_j \right \}\\
&=\dfrac{4(n-n_1)^2}{n^2(n-1)^2(n_1-1)^2}\bigg \{\sum_{i=1}^{l-1}\xi(\bi X^{(1)}_i ) +\sum_{1 \leq i \neq j \leq l-1}\eta(\bi X^{(1)}_i, \bi X^{(1)}_j)   \bigg \}.
\end{align*}
\noindent{\underline{\bf{Case 2}}}, for $n_1<l  \leq n_1+n_2$, we have 
\begin{align*}
&\sigma^2_{n,l} \\
&=\E\left \{ \bigg (-\dfrac{2(n-n_2)}{n(n-1)(n_2-1)}\sum^{l-n_1-1}_{j=1}d(\bi X_l, \bi X^{(2)}_j)+{n \choose 2}^{-1}\sum_{i=1}^{n_1}d(\bi X_l, \bi X^{(1)}_i)\bigg)^2| \mathscr{F}_{l-1} \right \}\\
&= \E\left \{ \bigg(-\dfrac{2(n-n_2)}{n(n-1)(n_2-1)}\bigg)\bigg (\sum^{l-n_1-1}_{i=1}d(\bi X_l, \bi X^{(2)}_i)+{n \choose 2}^{-1}\sum_{i=1}^{n_1}d(\bi X_l, \bi X^{(1)}_i\bigg)\right.\\
&\left. \bigg(-\dfrac{2(n-n_2)}{n(n-1)(n_2-1)}\bigg)\bigg (\sum^{l-n_1-1}_{j=1}d(\bi X_l, \bi X^{(2)}_j)+{n \choose 2}^{-1}\sum_{j=1}^{n_1}d(\bi X_l, \bi X^{(1)}_j\bigg)| \mathscr{F}_{l-1} \right \}\\
&=\dfrac{4(n-n_2)^2}{n^2(n-1)^2(n_2-1)^2} \sum_{i=1}^{l-n_1-1}\xi(\bi X^{(2)}_i)+\dfrac{4(n-n_2)^2}{n^2(n-1)^2(n_2-1)^2}  \sum_{1 \leq i \neq j \leq l-n_1-1}\eta(\bi X^{(2)}_i, \bi X^{(2)}_j)\\
&-\dfrac{8(n-n_2)}{ n^2(n-1)^2(n_2-1)}\sum_{i=1}^{n_1}\sum_{j=1}^{l-n_1-1}\eta(\bi X^{(1)}_i, \bi X^{(2)}_j)
+{n \choose 2}^{-2} \sum_{i=1}^{n_1}\xi(\bi X^{(1)}_i)\\
&+{n \choose 2}^{-2} \sum_{1 \leq i \neq j \leq n_1}\eta(\bi X^{(1)}_i, \bi X^{(1)}_j).
\end{align*}
\noindent{\underline{\bf{Case 3}}}, for  $n_1+n_2<l  \leq n$, we have 
\begin{align*}
&\sigma^2_{n,l}=
\E\left \{ \bigg (- \dfrac{2(n-n_3)}{n(n-1)(n_3-1)}\sum^{l-n_1-n_2-1}_{j=1}d(\bi X_l, \bi X^{(3)}_j)+{n \choose 2}^{-1}\sum_{i=1}^{n_1}d(\bi X_l, \bi X^{(1)}_i)\right.\\
&\left.+{n \choose 2}^{-1}\sum_{i=1}^{n_2}d(\bi X_l, \bi X^{(2)}_i)\bigg)^2| \mathscr{F}_{l-1} \right \}\\
&={n \choose 2}^{-2}\sum_{i=1}^{n_1}\xi(\bi X^{(1)}_i)+{n \choose 2}^{-2}\sum_{1 \leq i \neq j \leq n_1}\eta(\bi X^{(1)}_i, \bi X^{(1)}_j)+2{n \choose 2}^{-2}\sum_{i=1}^{n_1}\sum_{j=1}^{n_2}\eta(\bi X^{(1)}_i, \bi X^{(2)}_j)\\
&-2{n \choose 2}^{-1}\dfrac{2(n-n_3)}{n(n-1)(n_3-1)}\sum_{i=1}^{n_1}\sum_{j=1}^{l-n_1-n_2-1}\eta(\bi X^{(1)}_i, \bi X^{(3)}_j)+{n \choose 2}^{-2}\sum_{i=1}^{n_2}\xi(\bi X^{(2)}_i)\\
&+{n \choose 2}^{-2}\sum_{1 \leq i \neq j \leq n_2}\eta(\bi X^{(2)}_i, \bi X^{(2)}_j)
-2{n \choose 2}^{-1}\dfrac{2(n-n_3)}{n(n-1)(n_3-1)}\sum_{i=1}^{n_2}\sum_{j=1}^{l-n_1-n_2-1}\eta(\bi X^{(2)}_i, \bi X^{(3)}_j)\\
&+\dfrac{4(n-n_3)^2}{n^2(n-1)^2(n_3-1)^2}\sum_{i=1}^{l-n_1-n_2-1}\xi(\bi X^{(3)}_i)
+\dfrac{4(n-n_3)^2}{n^2(n-1)^2(n_3-1)^2}\\
& \times \sum_{1 \leq i \neq j \leq l-n_1-n_2-1}\eta(\bi X^{(3)}_i, \bi X^{(3)}_j).
\end{align*}
Therefore, under  independence of $\bi X$ and $Y$, we have 
\begin{align*}
&\E \Big(\sum_{l=1}^n \sigma^2_{n, l}\Big)=\dfrac{4(n-n_1)^2}{n^2(n-1)^2(n_1-1)^2}\sum_{l=1}^{n_1}\sum_{i=1}^{l-1}\E \Big(d(\bi X_l, \bi X^{(1)}_i)\Big)^2\\
&+\dfrac{4(n-n_2)^2}{n^2(n-1)^2(n_2-1)^2} \sum_{l=n_1+1}^{n_1+n_2}\sum_{i=1}^{l-n_1-1}\E \Big (d(\bi X_l, \bi X^{(2)}_i)\Big)^2\\
&+{n \choose 2}^{-2} \sum_{l=n_1+1}^{n_1+n_2}\sum_{i=1}^{n_1}\E \Big( d(\bi X_l, \bi X^{(1)}_i)\Big)^2+{n \choose 2}^{-2}\sum_{l=n_1+n_2+1}^{n} \sum_{i=1}^{n_1}\E \Big(d(\bi X_l, \bi X^{(1)}_i)\Big)^2\\
&+{n \choose 2}^{-2}\sum_{l=n_1+n_2+1}^{n}\sum_{i=1}^{n_2}\E \Big(d(\bi X_l, \bi X^{(2)}_i)\Big)^2\\
&+\dfrac{4(n-n_3)^2}{n^2(n-1)^2(n_3-1)^2}\sum_{l=n_1+n_2+1}^{n}\sum_{i=1}^{l-n_1-n_2-1}\E\Big(d(\bi X_l, \bi X^{(3)}_i)\Big)^2\\
&=\Big(\dfrac{2n_1(n-n_1)^2}{n^2(n-1)^2(n_1-1)}+\dfrac{2n_2(n-n_2)^2}{n^2(n-1)^2(n_2-1)}+\dfrac{2n_3(n-n_3)^2}{n^2(n-1)^2(n_3-1)}\\
&+\dfrac{4n_1n_2+4n_1n_3+4n_2n_3}{n^2(n-1)^2}\Big)\E d^2(\bi X_1, \bi X_2)\\
&=\dfrac{2}{n^2(n-1)^2}\bigg\{\dfrac{n_1(n-n_1)^2}{(n_1-1)}+\dfrac{n_2(n-n_2)^2}{(n_2-1)}+\dfrac{n_3(n-n_3)^2}{(n_3-1)}\\
&+2n_1n_2+2n_1n_3+2n_2n_3\bigg\}\E d^2(\bi X_1, \bi X_2). 
\end{align*}
It is not difficult to show that  
\begin{align} \label{Esigma}
\sigma_0^2 = \mbox{var}(G_{n})=\E \Big(\sum_{l=1}^n \sigma^2_{n,l}\Big).
\end{align}
To complete the proof of Lemma \ref{mct_1},  it suffices to show that 
\begin{align}\label{varratio}
\dfrac{var(\sum_{l=1}^n \sigma^2_{n,l})}{var^2(G_{n})} \to 0.
\end{align}
We partition $\sum_{l=1}^n \sigma^2_{n,l}$ into two parts, that is, 
\begin{align*}
\sum_{k=1}^n \sigma^2_{n,k}:=R^{(1)}_n+R^{(2)}_n,
\end{align*}
where 
\begin{align*}
&R^{(1)}_n=\dfrac{4(n-n_1)^2}{n^2(n-1)^2(n_1-1)^2}\sum_{k=1}^{n_1}\sum_{i=1}^{k-1}\xi(\bi X^{(1)}_i )+\dfrac{4(n-n_2)^2}{n^2(n-1)^2(n_2-1)^2}\sum_{k=n_1+1}^{n_1+n_2} \sum_{i=1}^{k-n_1-1}\xi(\bi X^{(2)}_i)
\end{align*}
\begin{align*}
&+{n \choose 2}^{-2} \sum_{k=n_1+1}^{n}\sum_{i=1}^{n_1}\xi(\bi X^{(1)}_i)+{n \choose 2}^{-2}\sum_{k=n_1+n_2+1}^{n}\sum_{i=1}^{n_2}\xi(\bi X^{(2)}_i)\\
&+\dfrac{4(n-n_3)^2}{n^2(n-1)^2(n_3-1)^2}\sum_{k=n_1+n_2+1}^{n}\sum_{i=1}^{k-n_1-n_2-1}\xi(\bi X^{(3)}_i),
\end{align*}
\begin{align*}
&R^{(2)}_n=\dfrac{4(n-n_1)^2}{n^2(n-1)^2(n_1-1)^2}\sum_{k=1}^{n_1}\sum_{1 \leq i \neq j \leq k-1}\eta(\bi X^{(1)}_i, \bi X^{(1)}_j)\\
&+\dfrac{4(n-n_2)^2}{n^2(n-1)^2(n_2-1)^2} \sum_{k=n_1+1}^{n_1+n_2} \sum_{1 \leq i \neq j \leq k-n_1-1}\eta(\bi X^{(2)}_i, \bi X^{(2)}_j)\\
&-\dfrac{8(n-n_2)}{ n^2(n-1)^2(n_2-1)}\sum_{k=n_1+1}^{n_1+n_2}\sum_{i=1}^{n_1}\sum_{j=1}^{k-n_1-1}\eta(\bi X^{(1)}_i, \bi X^{(2)}_j)\\
&+{n \choose 2}^{-2}\sum_{k=n_1+1}^n\sum_{1 \leq i \neq j \leq n_1}\eta(\bi X^{(1)}_i, \bi X^{(1)}_j)+2{n \choose 2}^{-2}\sum_{k=n_1+n_2+1}^n\sum_{i=1}^{n_1}\sum_{j=1}^{n_2}\eta(\bi X^{(1)}_i, \bi X^{(2)}_j)\\
&-2{n \choose 2}^{-1}\dfrac{2(n-n_3)}{n(n-1)(n_3-1)}\sum_{k=n_1+n_2+1}^n\sum_{i=1}^{n_1}\sum_{j=1}^{k-n_1-n_2-1}\eta(\bi X^{(1)}_i, \bi X^{(3)}_j)\\
&+{n \choose 2}^{-2}\sum_{k=n_1+n_2+1}^n\sum_{1 \leq i \neq j \leq n_2}\eta(\bi X^{(2)}_i, \bi X^{(2)}_j)\\
&-2{n \choose 2}^{-1}\dfrac{2(n-n_3)}{n(n-1)(n_3-1)}\sum_{k=n_1+n_2+1}^n\sum_{i=1}^{n_2}\sum_{j=1}^{k-n_1-n_2-1}\eta(\bi X^{(2)}_i, \bi X^{(3)}_j)\\
&+\dfrac{4(n-n_3)^2}{n^2(n-1)^2(n_3-1)^2}\sum_{k=n_1+n_2+1}^n\sum_{1 \leq i \neq j \leq k-n_1-n_2-1}\eta(\bi X^{(3)}_i, \bi X^{(3)}_j).
\end{align*}
Under independence of $\bi X$ and $Y$ and by the properties in Lemmas \ref{prop_d} and \ref{prop_eta}, $R^{(1)}$ and $R^{(2)}$ are orthogonal, that is,
$$ \E (R^{(1)}_n R^{(2)}_n)=0. $$
Also,
\begin{align*}
&\E (R^{(1)}_n)^2=\E \Big \{\dfrac{4(n-n_1)^2}{n^2(n-1)^2(n_1-1)^2}\sum_{k=1}^{n_1}\sum_{i=1}^{k-1}\xi(\bi X^{(1)}_i )+\dfrac{4(n-n_2)^2}{n^2(n-1)^2(n_2-1)^2}\\
&\times\sum_{k=n_1+1}^{n_1+n_2} \sum_{i=1}^{k-n_1-1}\xi(\bi X^{(2)}_i)
+{n \choose 2}^{-2} \sum_{k=n_1+1}^{n}\sum_{i=1}^{n_1}\xi(\bi X^{(1)}_i)
+{n \choose 2}^{-2}\sum_{k=n_1+n_2+1}^{n}\sum_{i=1}^{n_2}\xi(\bi X^{(2)}_i)\\
&+\dfrac{4(n-n_3)^2}{n^2(n-1)^2(n_3-1)^2}\sum_{k=n_1+n_2+1}^{n}\sum_{i=1}^{k-n_1-n_2-1}\xi(\bi X^{(3)}_i)\Big \}^2\\
\end{align*}
\begin{align*}
:=&\E\bigg\{A^2+B^2+C^2+D^2+E^2+2AB+2AC+2AD+2AE+2BC\\
&+2BD+2BE+2CD+2CE+2DE\bigg\},
\end{align*}
where
\begin{align*}
\E A^2&=\E \Big \{\dfrac{4(n-n_1)^2}{n^2(n-1)^2(n_1-1)^2}\sum_{k=1}^{n_1} \sum_{i=1}^{k-1}\xi(\bi X^{(1)}_i)\Big\}^2\\
&=\dfrac{16(n-n_1)^4}{n^4(n-1)^4(n_1-1)^4}\Big\{\dfrac{(n_1-1)(2n_1-1)n_1}{6}\gamma^4+\Big(\dfrac{n_1(n_1-1)^2(n_1-2)}{4}\\
&+\dfrac{n_1(n_1-1)(n_1-2)}{6}\Big)\sigma^4\Big\},\\
\E B^2&=\E \Big \{\dfrac{4(n-n_2)^2}{n^2(n-1)^2(n_2-1)^2}\sum_{k=n_1+1}^{n_1+n_2} \sum_{i=1}^{k-n_1-1}\xi(\bi X^{(2)}_i)\Big\}^2\\
&=\dfrac{16(n-n_2)^4}{n^4(n-1)^4(n_2-1)^4}\Big\{\dfrac{(n_2-1)(2n_2-1)n_2}{6}\gamma^4+\Big(\dfrac{n_2(n_2-1)^2(n_2-2)}{4}\\
&+\dfrac{n_2(n_2-1)(n_2-2)}{6}\Big)\sigma^4\Big\},\\
\E C^2&=\E\big( {n \choose 2}^{-2} \sum_{k=n_1+1}^{n_1+n_2}\sum_{i=1}^{n_1}\xi(\bi X^{(1)}_i)\big)^2={n \choose 2}^{-4} (n_2+n_3)^2\{n_1 \gamma^4+n_1(n_1-1)\sigma^4\},\\
\E D^2&=\E\big( {n \choose 2}^{-2}\sum_{k=n_1+n_2+1}^{n}\sum_{i=1}^{n_2}\xi(\bi X^{(2)}_i)\big)^2={n \choose 2}^{-4}n^2_3\{n_2 \gamma^4+n_2(n_2-1)\sigma^4\},\\
\E E^2&=\E \Big \{\dfrac{4(n-n_3)^2}{n^2(n-1)^2(n_3-1)^2}\sum_{k=n_1+n_2+1}^{n} \sum_{i=1}^{k-n_1-n_2-1}\xi(\bi X^{(3)}_i)\Big\}^2\\
&=\dfrac{16(n-n_3)^4}{n^4(n-1)^4(n_3-1)^4}\Big\{\dfrac{(n_3-1)(2n_3-1)n_3}{6}\gamma^4+\Big(\dfrac{n_3(n_3-1)^2(n_3-2)}{4}\\
&+\dfrac{n_3(n_3-1)(n_3-2)}{6}\Big)\sigma^4\Big\},\\
\E AB&=\dfrac{16 (n-n_1)^2(n-n_2)^2}{n^4(n-1)^4(n_1-1)^2(n_2-1)^2}\dfrac{n_1(n_1-1)}{2}\dfrac{n_2(n_2-1)}{2}\sigma^4,\\
\E AC&={n \choose 2}^{-2}\dfrac{4(n-n_1)^3}{n^2(n-1)^2(n_1-1)^2}\Big(\dfrac{n_1(n_1-1)}{2}\gamma^4+\big(\dfrac{n_1(n_1-1)(n_1-2)}{2}+\dfrac{n_1(n_1-1)}{2}\big)\sigma^4\Big),\\
\E AD&=\dfrac{4(n-n_1)^2}{n^2(n-1)^2(n_1-1)^2} {n \choose 2}^{-2} n_3n_2 \dfrac{n_1(n_1-1)}{2}\sigma^4,\\
\E AE&=\dfrac{4(n-n_1)^2}{n^2(n-1)^2(n_1-1)^2}\dfrac{4(n-n_3)^2}{n^2(n-1)^2(n_3-1)^2}\dfrac{n_1(n_1-1)}{2}\dfrac{n_3(n_3-1)}{2}\sigma^4,
\end{align*}
\begin{align*}
\E BC&=\dfrac{4(n-n_2)^2}{n^2(n-1)^2(n_2-1)^2} {n \choose 2}^{-2} n_1(n-n_1) \dfrac{n_2(n_2-1)}{2}\sigma^4,\\
\E BD&={n \choose 2}^{-2}\dfrac{4n_3(n-n_2)^2}{n^2(n-1)^2(n_2-1)^2}\Big(\dfrac{n_2(n_2-1)}{2}\gamma^4+\big(\dfrac{n_2(n_2-1)(n_2-2)}{2}+\dfrac{n_2(n_2-1)}{2}\big)\sigma^4\Big),\\
\E BE&=\dfrac{4(n-n_2)^2}{n^2(n-1)^2(n_2-1)^2}\dfrac{4(n-n_3)^2}{n^2(n-1)^2(n_3-1)^2}\dfrac{n_2(n_2-1)}{2}\dfrac{n_3(n_3-1)}{2}\sigma^4,\\
\E CD&={n \choose 2}^{-4}n_1n_2n_3(n_2+n_3)\sigma^4,\\
\E CF&=\dfrac{4(n-n_3)^2}{n^2(n-1)^2(n_3-1)^2}{n \choose 2}^{-2}n_1(n_2+n_3)\dfrac{n_3(n_3-1)}{2}\sigma^4,\\
\E DE&=\dfrac{4(n-n_3)^2}{n^2(n-1)^2(n_3-1)^2}{n\choose 2}^{-2}n_2n_3 \dfrac{n_3(n_3-1)}{2}\sigma^4.
\end{align*}
Therefore,
\begin{align*}
&\E (R^{(1)}_n)^2=\dfrac{4}{n^4(n-1)^4}\Big \{ \dfrac{n^2_1(n-n_1)^4}{(n_1-1)^2}+\dfrac{n^2_2(n-n_2)^4}{(n_2-1)^2}+\dfrac{n^2_3(n-n_3)^4}{(n_3-1)^2}\\
&+4 n^2_1n^2_2+4 n^2_1n^2_3+4 n^2_2n^2_3+8n^2_1n_2n_3+8n_1n^2_2n_3+8n_1n_2n^2_3\\
&+\dfrac{2n_1n_2(n-n_1)^2(n-n_2)^2}{(n_1-1)(n_2-1)}+\dfrac{2n_1n_3(n-n_1)^2(n-n_3)^2}{(n_1-1)(n_3-1)}+\dfrac{2n_2n_3(n-n_2)^2(n-n_3)^2}{(n_2-1)(n_3-1)}\\
&+\dfrac{4n^2_1n_2+4n^2_1n_3+4n_1n_2n_3}{n_1-1}(n-n_1)^2\\
&+\dfrac{4n_1n^2_2+4n^2_2n_3+4n_1n_2n_3}{n_2-1}(n-n_2)^2\\
&+\dfrac{4n_1n^2_3+4n_2n^2_3+4n_1n_2n_3}{n_3-1}(n-n_3)^2+o(n^{-4})
\Big \}\sigma^4+O(n^{-5})\gamma^4\\
&=(16n^{-4}+o(n^{-4})\sigma^4+O(n^{-5})\gamma^4.
\end{align*}
Similarly, after a tedious evaluation, we have
\begin{align*}
\E (R^{(2)}_n)^2&=\tau^4 {n \choose 2}^{-4} \Big \{ n^1_1(n_2+n_3)^2+n^2_3n^2_2+4n_1n_2n^2_3+\dfrac{n_1(n_2+n_3)^3}{3}\\
&+\dfrac{n_2(n_1+n_3)^3}{3}-4n_1n_2n_3(n_1+n_3)\Big\}+o(n^{-4})\\
&=O(n^{-4})\tau^{4}+o(n^{-4}).
\end{align*}
Now we have 
\begin{align*}
\mbox{var}(\sum_{k=1}^n\sigma^2_{n,k})&=\E(\sum_{k=1}^n\sigma^2_{n,k})^2-\{\E(\sum_{k=1}^n\sigma^2_{n,k})\}^2 =\E (R^{(1)}_n)^2+\E (R^{(2)}_n)^2-\mbox{var}^2(G_n).
\end{align*}
To prove (\ref{varratio}),  we only need to show that
\begin{align*}
\dfrac{\E(\sum^n_{l=1}\sigma^2_{n,l})^2}{var^2(G_n)} \to 1.
\end{align*}
This is true, because 
\begin{align*}
&\E(\sum^n_{k=1}\sigma^2_{n,k})^2=\E (R^{(1)}_n)^2+\E (R^{(2)}_n)^2\\
&=\dfrac{4}{n^4(n-1)^4}\Big \{ \dfrac{n^2_1(n-n_1)^4}{(n_1-1)^2}+\dfrac{n^2_2(n-n_2)^4}{(n_2-1)^2}+\dfrac{n^2_3(n-n_3)^4}{(n_3-1)^2}\\
&+4 n^2_1n^2_2+4 n^2_1n^2_3+4n^2_2n^2_3+8n^2_1n_2n_3+8n_1n^2_2n_3+8n_1n_2n^2_3\\
&+\dfrac{2n_1n_2(n-n_1)^2(n-n_2)^2}{(n_1-1)(n_2-1)}+\dfrac{2n_1n_3(n-n_1)^2(n-n_3)^2}{(n_1-1)(n_3-1)}+\dfrac{2n_2n_3(n-n_2)^2(n-n_3)^2}{(n_2-1)(n_3-1)}\\
&+\dfrac{4n^2_1n_2+4n^2_1n_3+4n_1n_2n_3}{n_1-1}(n-n_1)^2
+\dfrac{4n_1n^2_2+4n^2_2n_3+4n_1n_2n_3}{n_2-1}(n-n_2)^2\\
&+\dfrac{4n_1n^2_3+4n_2n^2_3+4n_1n_2n_3}{n_3-1}(n-n_3)^2
\Big \}\sigma^4+O(n^{-5})\gamma^4+o(n^{-4})+O(n^{-4})\tau^4\\
&=\dfrac{16  \sigma^4}{n^4}+o(1),
\end{align*}
where the last equality is obtained under conditions \textbf{C}2 and \textbf{C}3 and Lemma (\ref{ncond}).  From (\ref{eqn:vargcov}), we have
\begin{align*}
\mbox{var}^2(G_n)&=\dfrac{4}{n^4}\bigg(\dfrac{n_1}{n_1-1}+\dfrac{n_2}{n_2-1}+\dfrac{n_3}{n_3-1}-\dfrac{n}{n-1}\bigg)^2\sigma^4\\
&=\dfrac{16\sigma^4}{n^4}+o(1).
\end{align*}
Therefore,  as $\min\{n_1, n_2, n_3\} \to \infty$, 
\begin{align*}
\dfrac{\E(\sum^n_{l=1}\sigma^2_{n,l})^2}{\mbox{var}^2(G_{n})} \to 1 \;\;\;\;\mbox{ and} \;\;\;\dfrac{\mbox{var}(\sum^n_{l=1}\sigma^2_{n,l})}{\mbox{var}^2(G_{n})} \to 0.
\end{align*}
The last step of the proof is to apply Chebyshev's inequality together with (\ref{Esigma}) and (\ref{varratio}). More specifically, for any $\varepsilon >0$, 
\begin{align*}
& P\left(\left|\frac{\sum_{l=1}^n \sigma_{n,l}^2}{\sigma_0^2}-1\right| > \varepsilon\right) = P\left(\left |\sum_{l=1}^n\sigma_{n,l}^2 -\E \left(\sum_{l=1}^n\sigma_{n,l}^2\right)\right | > \varepsilon \mbox{var}(G_n)\right)\\
& \leq \frac{\mbox{var}(\sum_{l=1}^n \sigma_{n,l}^2)}{\epsilon^2 \mbox{var}^2(G_n)} \rightarrow 0.
\end{align*} 
This completes the proof for Lemma \ref{mct_1}.
\end{proof}
\begin{lemma}\label{mct_2}
Under conditions \textbf{C}1-\textbf{C}2 and  independence of $\bi X$ and $Y$, as $\min\{n_1, n_2, n_3\} \to \infty $, we have 
\begin{align*}
\dfrac{\sum_{l=1}^n \E(M^4_{n,l})}{\mbox{var}^2(G_{n})} \to 0.
\end{align*}
\end{lemma}
\begin{proof} Now we compute $\E M^4_{n, l}$ under  independence of $\bi X$ and $Y$. 

\noindent{\underline{\bf Case 1}}, for $1\leq l \leq n_1$, we have 
\begin{align*}
&\E M^4_{n,l}=\E\Big\{-\dfrac{2(n-n_1)}{n(n-1)(n_1-1)}\sum^{l-1}_{j=1}d(\bi X_l, \bi X^{(1)}_j)\Big\}^4\\
&=\dfrac{16(n-n_1)^4}{n^4(n-1)^4(n_1-1)^4}\{(l-1)\E d^4(\bi X_1, \bi X_l)+3(l-1)(l-2)\E d^2(\bi X_1, \bi X_l)d^2(\bi X_2, \bi X_l)\}\\
&=\dfrac{16(n-n_1)^4}{n^4(n-1)^4(n_1-1)^4}\{(l-1)\omega^4+3(l-1)(l-2)\gamma^4\};
\end{align*}
\noindent{\underline{\bf Case 2}}, for $n_1<l \leq n_1+n_2$, we have
\begin{align*}
&\E M^4_{n,l}=\E\Big\{-\dfrac{2(n-n_2)}{n(n-1)(n_2-1)}\sum^{l-n_1-1}_{j=1}d(\bi X_l, \bi X^{(2)}_j)+{n \choose 2}^{-1}\sum_{i=1}^{n_1}d(\bi X_l, \bi X^{(1)}_i)\Big\}^4\\
&=\E A^4+\E B^4+6\E A^2B^2,
\end{align*}
where 
\begin{align*}
&\E A^4=\dfrac{16(n-n_2)^4}{n^4(n-1)^4(n_2-1)^4} \{(l-n_1-1) \omega^4+3(l-n_1-1)(l-n_1-2)\gamma^4\},\\
&\E B^4={n \choose 2}^{-4} \{n_1\omega^4+3n_1(n_1-1)\gamma^4\},\\
&\E A^2B^2=\Big(-\dfrac{2(n-n_2)}{n(n-1)(n_2-1)}\Big)^2{n \choose 2}^{-2}\{n_1(l-n_1-1)\}\gamma^4.
\end{align*}
\noindent{\underline{\bf Case 3}}, for $n_1+n_2<l \leq n$, we have
\begin{align*}
\E M^4_{n,l}&=\E\Big\{- \dfrac{2(n-n_3)}{n(n-1)(n_3-1)}\sum^{l-n_1-n_2-1}_{j=1}d(\bi X_l, \bi X^{(3)}_j)+{n \choose 2}^{-1}\sum_{i=1}^{n_1}d(\bi X_l, \bi X^{(1)}_i)\\
&+{n \choose 2}^{-1}\sum_{i=1}^{n_2}d(\bi X_l, \bi X^{(2)}_i)\Big\}^4\\
&=\E A^4+\E B^4+ \E C^4+6(\E A^2B^2+\E A^2C^2+\E B^2C^2),
\end{align*}
where 
\begin{align*}
\E A^4&=\dfrac{16(n-n_3)^4}{n^4(n-1)^4(n_3-1)^4}\{(l-n_1-n_2-1) \omega^4+3(l-n_1-n_2-1)\\
&\times(l-n_1-n_2-2)\gamma^4\},\\
\E B^4&={n \choose 2}^{-4} \{n_1\omega^4+3n_1(n_1-1)\gamma^4\},
\end{align*}
\begin{align*}
&\E C^4={n \choose 2}^{-4} \{n_2 \omega^4+3n_2(n_2-1)\gamma^4\},\\
&\E A^2B^2=\Big(-\dfrac{2(n-n_3)}{n(n-1)(n_3-1)}\Big)^2{n \choose 2}^{-2}\{n_1(l-n_1-n_2-1)\}\gamma^4,\\
&\E A^2C^2=\Big(-\dfrac{2(n-n_3)}{n(n-1)(n_3-1)}\Big)^2{n \choose 2}^{-2}\{n_2(l-n_1-n_2-1)\}\gamma^4,\\
&\E B^2C^2={n \choose 2}^{-4}n_1n_2\gamma^4.
\end{align*}
Therefore,
\begin{align*}
&\sum_{l=1}^n \E M^4_{n,l}=\sum_{l=1}^{n_1}\dfrac{16(n-n_1)^4}{n^4(n-1)^4(n_1-1)^4}\{(l-1)\omega^4+3(l-1)(l-2)\gamma^4\}\\
&+\sum_{l=n_1+1}^{n_2} \Big\{\dfrac{16(n-n_1)^4}{n^4(n-1)^4(n_1-1)^4} \{(l-n_1-1) \omega^4+3(l-n_1-1)(l-n_1-2)\gamma^4\}\\
&+{n \choose 2}^{-4} \{n_1\omega^4+3n_1(n_1-1)\sigma^4\}+6\Big(\dfrac{2(n-n_2)}{n(n-1)(n_2-1)}\Big)^2{n \choose 2}^{-2}\{n_1(l-n_1-1)\}\gamma^4\Big\}\\
&+\sum_{l=n_1+n_2+1}^n \Big\{ \dfrac{16(n-n_3)^4}{n^4(n-1)^4(n_3-1)^4}\{(l-n_1-n_2-1) \omega^4+3(l-n_1-n_2-1)\\
&(l-n_1-n_2-2)\gamma^4\}\\
&+{n \choose 2}^{-4} \{n_1\omega^4+3n_1(n_1-1)\sigma^4\}+{n \choose 2}^{-4} \{n_2 \omega^4+3n_2(n_2-1)\gamma^4\}\\
&+6\Big(-\dfrac{2(n-n_3)}{n(n-1)(n_3-1)}\Big)^2{n \choose 2}^{-2}\{n_1(l-n_1-n_2-1)\}\gamma^4\\
&+6\Big(\dfrac{2(n-n_3)}{n(n-1)(n_3-1)}\Big)^2{n \choose 2}^{-2}\{n_2(l-n_1-n_2-1)\}\gamma^4+6{n \choose 2}^{-4}n_1n_2\gamma^4\\
&=O(n^{-5})\gamma^4+O(n^{-5})\omega^4\\
&=o(n^{-4})\sigma^4.
\end{align*}
The last equality is due to Condition \textbf{C}2 and Lemma \ref{ncond}. This completes the proof of this lemma. 
\end{proof}

Lemma \ref{mct_2} implies that the Lindeberg's condition holds. Along with Lemma \ref{mct_1}, an application of the martingale CLT completes the proof of Theorem \ref{gcov1}. 

\subsection{Proof of Theorem \ref{power}} 

As $\hat{\sigma}_0$ in (\ref{eqn:sigma0}) is a ratio consistent estimator for $\sigma_0$, it is sufficient to show that $\dfrac{\mbox{gCov}_n(X, Y)}{\sigma_0}>C$ for any arbitrarily large constant $C>0$ under ${\cal H}_1$. 

\begin{align}
&\E \big(\mbox{gCov}_n(\bi X, Y)-\mbox{gCov}(\bi X, Y)\big)^2 \nonumber\\
&=\E \bigg((U_n-\E U_n)-\sum_{k=1}^K \hat{p}_k (U_{n_k}-\E U_{n_k})+\sum_{k=1}^K (p_k-\hat{p}_k) \E U_{n_k}\bigg)^2\nonumber\\
&\leq (2K+1)\left\{ \E (U_n-\E U_n)^2+\sum_{k=1}^K \hat{p}^2_k \E(U_{n_k}-\E U_{n_k})^2+\sum_{k=1}^K\E (p_k-\hat{p}_k)^2( \E U_{n_k})^2\right\}\nonumber\\
&\leq (2K+1) \left( \dfrac{C_1}{n}\E\|\bi X_i-\bi X_j\|^2+C_2\sum_{k=1}^K \dfrac{ \hat{p}^2_k}{n_k} \E\|\bi X^{(k)}_i-\bi X^{(k)}_j\|^2 +\sum_{k=1}^K\frac{p_k(1-p_k)\Delta_k^2}{n_k} \right)\label{ineq:sigma1}\\
&= \dfrac{C(2K+1)}{n}\bigg(\E\|\bi X_i-\bi X_j\|^2+\sum_{k=1}^K \hat{p}_k \E\|\bi X^{(k)}_i-\bi X^{(k)}_j\|^2+O(1)\bigg).\nonumber
\end{align}
The inequality (\ref{ineq:sigma1}) is obtained by applying the moment inequality of $U$-statistics from \cite{Koroljuk1994} (p.72) and conditional Jensen's inequality. Hence,
\begin{align*}
|\mbox{gCov}_n(\bi X, Y)-\mbox{gCov}(\bi X, Y)| = O_p(n^{-1/2}).
\end{align*}
With the equation (\ref{eqn:vargcov1}), we have
\begin{align} \label{order1}
\left|\dfrac{\mbox{gCov}_n(\bi X, Y)}{\sigma_0}-\dfrac{\mbox{gCov}(\bi X, Y)}{\sigma_0}\right|= O_p(n^{1/2}) \rightarrow \infty. 
\end{align}
Under condition \textbf{C}4, $\sqrt{n}\mbox{gCov}(\bi X, Y)\rightarrow \infty$,  we have 
\begin{align}\label{order2}
\left| \dfrac{\mbox{gCov}_n(\bi X, Y)-\mbox{gCov}(\bi X, Y)}{\mbox{gCov}(\bi X, Y)}\right| \rightarrow 0 \;\; \mbox{in probability}. 
\end{align}
With (\ref{order1}) and (\ref{order2}) together, we can conclude that  $\dfrac{\mbox{gCov}_n(\bi X, Y)}{\sigma_0}  \rightarrow \infty$ in probability. Therefore, $P(\mbox{gCov}_n(\bi X, Y)>Z_\alpha \hat{\sigma}_0) \rightarrow 1$.  We have completed the proof. 

\vspace{0.5cm}
\noindent{\bf Acknowledgement} 

\vspace{0.2cm}
\noindent Thanks Jun Li for sharing her R codes on two-sample tests with us.

\end{document}